\documentclass[graybox]{svmult}
\usepackage{theorem}

\usepackage{mathptmx, amsmath}       % selects Times Roman as basic font
\usepackage{helvet}         % selects Helvetica as sans-serif font
\usepackage{courier}        % selects Courier as typewriter font
\usepackage{type1cm}        % activate if the above 3 fonts are
\usepackage{makeidx}         % allows index generation
 \usepackage {graphicx}      % standard LaTeX graphics toolx
\usepackage{multicol}        % used for the two-column index
\usepackage[bottom]{footmisc}% places footnotes at page bottom

\usepackage{amsfonts, times,pdfpages}

\usepackage{appendix}
\newtheorem{assumption}{Assumption}
\numberwithin{equation}{section}

\numberwithin{equation}{section}
\newtheorem{lem}{Lemma}[section]
\newtheorem{prop}[lem]{Proposition}
\newtheorem{thm}[lem]{Theorem}

{
\theorembodyfont{\rmfamily}
\newtheorem{rem}[lem]{Remark}

\newtheorem{defn}[lem]{Definition}
\newtheorem{cor}[lem]{Corollary}
}

\newtheorem{exa}[lem]{Example}

\def\paral{/\kern-0.55ex/}
\def\parals_#1{/\kern-0.55ex/_{\!#1}}
\def\bparals_#1{\breve{/\kern-0.55ex/_{\!#1}}}
\def\n#1{|\kern-0.24em|\kern-0.24em|#1|\kern-0.24em|\kern-0.24em|}

\def\EE{{\mathbf E}}

\newcommand{\F}{{\mathcal F}}

\def\g{{\mathfrak g}}

\def\h{{\mathfrak h}}

\def\m{{\mathfrak m}}

\newcommand{\R}{{\mathbb R}}

\def\N{{\mathcal N}}

\def\L{{\mathcal L}}
\def\1{{\mathbf 1}}

\def\f{\frac}

\def\ad{{\mathop {\rm ad}}}

\def\Ad{{\mathop {\rm Ad}}}

\def\s.t.{\mathop {\rm s.t.}}

\def\div{\mathop{\rm div}}

\def\ker{\mathop{\rm ker}}

\def\Dom{\mathop{\rm Dom}}

\def\index{\mathop{\rm index}}
\def\range{\mathop{\rm Range}}

\def\so{{\mathfrak {so}}}

\def\<{\langle}
\def\>{\rangle}

\def\Lip{\mathrm {Lip}}
\def\Iso{\mathrm {Iso}}

\def\paral{/\kern-0.55ex/}
\def\parals_#1{/\kern-0.55ex/_{\!#1}}
\def\bparals_#1{\breve{/\kern-0.55ex/_{\!#1}}}
\def\n#1{|\kern-0.24em|\kern-0.24em|#1|\kern-0.24em|\kern-0.24em|}
\begin{document}

\title*{  Perturbation of Conservation Laws\\
and Averaging on Manifolds}
\author{Xue-Mei Li}
\institute{Xue-Mei Li \at Department of Mathematics,  Imperial College London, London SW7 2AZ, U.K. \email{xue-mei.li@imperial.ac.uk}}

\maketitle

\abstract{ We prove a stochastic averaging theorem for stochastic differential equations in which
 the slow and the fast variables interact. The approximate Markov fast motion is a family of Markov process
 with generator $\L_x$ for which we obtain a locally uniform law of large numbers and obtain the continuous dependence of their invariant measures on the parameter $x$.  These results are obtained  under the assumption that  $\L_x$ satisfies H\"ormander's bracket conditions, or more generally $\L_x$ is a family of Fredholm operators with sub-elliptic estimates. 
 On the other hand a conservation law of a dynamical system can be used as a tool for separating the scales in  singular perturbation problems. We also study a number of motivating examples from mathematical physics and from geometry where we use non-linear  conservation laws  to deduce   slow-fast systems of stochastic differential equations.}

\bigskip

\begin{center} {\it Table of Contents}\end{center}

{{1.} Introduction} 

\quad \quad{ {1.1} \; Description of results}

{ {2.} Examples}

\quad\quad {2.1} \; A dynamical description for Brownian motions

\quad\quad {2.2} \; Collapsing of manifolds

\quad\quad {2.3} \; Inhomogeneous scaling of Riemannian metrics

\quad\quad {2.4} \; Perturbed dynamical systems on principal bundles

\quad\quad {2.5} \; Completely integrable stochastic Hamiltonian systems

 {3.} Ergodic theorem for Fredholm operators depending on a parameter

 {4.} Basic Estimates for SDEs on manifolds
 
{5.} Proof of Theorem 2

{6.} Re-visit the examples
%
%\quad\quad {6.1} \; A dynamical description for hypo-elliptic diffusions
%
%\quad\quad {6.2} \; Inhomogeneous scaling of Riemannian metrics
%
%\quad\quad {6.3} \; An averaging principle on Principal bundles

{References}

\section{Introduction}

A deterministic or random system with a conservation law is often used to  approximate the motion of an object that is also subjected to many other smaller deterministic or random influences.
The latter is a perturbation of the former.
To describe the evolution of the  dynamical system, we  begin with these conservation laws.
 A conservation law is  a quantity which does not change with time, for us it  is an  equi-variant map  on a manifold, i.e. a map which is  invariant under an action of  a group.
  They describe the orbit of the action. Quantities describing the perturbed systems have their natural scales, the conservations laws can be used to determine the different components  of the system which evolve at different speeds. Some components may move at a much faster speed than some others, in which case we either ignore the slow components,  in other words we  approximate  the perturbed system by the unperturbed one,   or  ignore the fast components and describe
  the slow components for which the  key ingredient is ergodic averaging.    It is a standard assumption that the fast variable moves so fast that its influence averaged over any time interval, of the size comparable to the natural scale of our observables, is effectively  that of an averaged vector field. The averaging is with respect to a probability measure on the state space of the fast variable.
 Depending on the object of the study, we will need to neglect either the small perturbations or quantities too large (infinities) to fit into the natural scale of things. To study singularly perturbation operators, we must  discard the infinities and at the same time retain the relevant information on the natural scale.
   In Hamiltonian formulation, for example,  the time evolution of an object, e.g. the movements of celestial bodies, is governed by a Hamiltonian function. 
If the magnitude of the Hamiltonian is set to be of order `1',  the  magnitude of the perturbation (the collective negligible influences) is of order  $\epsilon$, then the perturbation is negligible on an interval of any fixed length. This ratio in magnitudes translates into  time scales. If the original system is on scale $1$,    we work on a time interval of length $\f 1 \epsilon$  to see the deviation of the perturbed trajectories. 
 Viewed on the time interval $[0,1]$ the perturbation is not observable. On $[0, \f 1 \epsilon]$ the perturbation is observable, the natural object to study is the evolution of the energies while  the dynamics of the Hamiltonian dynamics becomes too large.  See \cite{Freidlin-Wentzell-original, Arnold89,  Borodin-Freidlin}.
 
 \bigskip

  If the state space of our dynamical system has an action by a group,  the orbit manifold is a fundamental object. We use the projection to the orbit manifold as a conservation law and use it to separate the slow and the fast variables in the system. The slow variables lie naturally on a quotient manifold.  In many examples we can further reduce this system of slow-fast stochastic differential equations (SDEs) to a product manifold $N\times G$,  which we describe later by examples. From here we proceed to prove an averaging principle for  the family of  SDEs with a parameter~$\epsilon$.  In these SDEs the slow and the fast variables are already separate, but they interact with each other. 
  
  This can then  be applied to  a local product space such as a principal bundle.
  In  \cite{Li-geodesic, Li-homogeneous, Li-limits}, the slow variables in the reduced system are random ODEs, where we study the system on the scale of $[1, \f 1 {\epsilon^2}]$ to obtain results of the nature of diffusion creation. In these studies we bypassed stochastic averaging and went straight for the diffusion creation. 
  In \cite{Li-averaging, Hogele-Ruffino, Gonzales-Gargate-Ruffino} stochastic averaging are studied, but they are computed in local coordinates.
  Here the slow variables  solve a genuine  SDE with  a stochastic integral and the computations are global.
 We first prove an averaging theorem for these SDEs and then study some examples where we deduce a slow-fast  system of SDEs from non-linear conservation laws, to which our main theorems apply.
\bigskip

Throughout the article $(\Omega, \F, \F_t, P)$ is a probability space satisfying the usual assumptions.
Let $(B_t, W_t)$ be a Brownian motion on $\R^{m_1}\times \R^{m_2}$ where $m_1, m_2\in \N$. We write
 $B_t=(B_t^1, \dots, B_t^{m_1})$ and $W_t=(W_t^1, \dots, W_t^{m_1})$.
 Let $N$ and $G$ be two complete connected smooth Riemannian manifolds, let $x_0\in N$ and $y_0\in G$. Let $\epsilon$ denote a small positive number and let  $m_1$ and $ m_2$ be two natural numbers. Let $X: N\times G\times \R^{m_1} \to TN$ 
and $Y: N\times G\times \R^{m_2} \to TG$ be $C^3$ smooth maps linear in the last variable. Let $X_0$ and $Y_0$ be $C^2$ smooth vector fields on $N$ and on $G$ respectively, with a parameter taking its values in the other manifold. Let us consider the SDEs, 
\begin{equation}\label{slow-fast-sdes}
\left\{\begin{aligned} dx_t^\epsilon=& X(x_t^\epsilon, y_t^\epsilon)\circ dB_t+
X_0(x_t^\epsilon, y_t^\epsilon) \,dt, \quad &x_0^\epsilon=x_0,\\
dy_t^\epsilon=& \f 1 {\sqrt\epsilon}Y(x_t^\epsilon,y_t^\epsilon) \circ dW_t+
\f 1{\epsilon} Y_0(x_t^\epsilon, y_t^\epsilon)\,dt, \quad &y_0^\epsilon=y_0.\end{aligned}\right.
\end{equation}
The symbol $\circ $ is used to  denote Stratonovich integrals.
By choosing an orthonormal basis $\{e_i\}$ of $\R^{m_1}\times \R^{m_2}$, we obtain a family of
 vector fields $\{X_1, \dots, X_{m_1}, Y_1, \dots, Y_{m_2}\}$ as following:
  $X_i(x)=X(x)(e_i)$ for $1\le i\le m_1$
 and $Y_i(x)=Y(x)(e_i)$ for $i=m_1+1, \dots, m_1+m_2$. Then the system of SDEs (\ref{slow-fast-sdes}) is equivalent to the following
 $$\left\{\begin{aligned} dx_t^\epsilon=&\sum_{k=1}^{m_1}  X_k(x_t^\epsilon, y_t^\epsilon)\circ dB_t^k+
X_0(x_t^\epsilon, y_t^\epsilon) \,dt, \quad   &x_0^\epsilon=x_0,\\
dy_t^\epsilon=& \f 1 {\sqrt\epsilon}\sum_{k=1}^{m_2}  Y_k(x_t^\epsilon, y_t^\epsilon) \circ dW_t^k+
\f 1{\epsilon} Y_0(x_t^\epsilon, y_t^\epsilon)\,dt,\quad &y_0^\epsilon=y_0. \end{aligned}\right.$$

If $V$ is a vector field, by  $Vf$ we mean $df(V)$ or $L_Vf$, the Lie differential of $f$ in the direction of $V$. Then $(x_t^\epsilon, y_t^\epsilon)$ is a sample continuous Markov process with generator $\L^\epsilon :=\f 1{ \epsilon} \L+\L^{(1)}$ where
$$\L=\f 12 \sum_{k=1}^{m_2} Y_i ^2+Y_0, \quad \L^{(1)}=\f 12 \sum_{k=1}^{m_1} X_i^2+X_0.$$

In other words if $f:N\times G\to \R$ is a smooth function then $$\L^\epsilon f(x,y):=\f 1{ \epsilon} \L_x (f (\cdot, y))(x)+\L_y^{(1)}( f(x, \cdot))(y),$$
where
$$\begin{aligned}\L_x f(x, \cdot) &=\left( \f 12 \sum_{k=1}^{m_2} Y_i^2(x,\cdot) +Y_0(x, \cdot)\right)f(x,\cdot), \\ \L_y^{(1)}f (\cdot,y)&=\left( \f 12 \sum_{k=1}^{m_1} X_i^2 (\cdot, y)+X_0(\cdot, y) \right)f(\cdot, y).\end{aligned}$$

The result we seek is the weak convergence of the slow variables $x_t^\epsilon$ to a Markov process $\bar x_t$ whose  Markov generator $\bar \L$ is to be described. 
\bigskip

Let $T$ be a positive number and let $C([0,T];N)$ denote the family continuous functions  from $[0, T]$ to   $N$,  the topology on $C([0,T];N)$  is given by the uniform distance.  A family of  continuous stochastic processes $x_t^\epsilon$ on $N$ is said to converge to a continuous process $\bar x_t$ if for every bounded continuous function $F: C([0,T];N)\to \R$,  as $\epsilon$ converges to zero,
$$\EE [F(x_\cdot^\epsilon)]\to \EE[ F(\bar x_\cdot)].$$

In particular, if  $u^\epsilon(t,x,y)$ is a bounded regular  solution to the Cauchy problem for the PDE ( for example $C^3$ in space and $C^1$ in time) $\f {\partial u^\epsilon} {\partial t} =\L^\epsilon u$
with the initial value $f$ in $L_\infty$,   then
$ u^\epsilon(t,x_0,y_0)=\EE[ f(x_t^\epsilon, y_t^\epsilon) ]$.   Suppose that the initial value function $f$ is  independent of the second variable so $f:N\to \R$. Then  the weak convergence will imply that $$\lim_{\epsilon \to 0} u^\epsilon(t,x_0,y_0)=u(t,x_0)$$ where $u(t,x)$ is the bounded regular solution to  the Cauchy problem $$\f {\partial u} {\partial t} =\bar  \L u,  \quad u(0,x)=f(x).$$
 \medskip
  
Stochastic averaging is a procedure of equating time averages with space averages using  a form of Birkhoff's ergodic theorem or a law of large numbers. 
Birkhoff's pointwise ergodic theorem states that if $T:E\to E$  is a measurable transformation  preserving a probability measure $\mu$ on the metric space $E$ then for any $F\in L^1(\mu)$, 
$$\f 1 n\sum_{k=1}^n F(T^kx)\to \EE (F| {\mathcal I})$$
for almost surely all $x$, as $n\to \infty$,  and where ${\mathcal I}$  is the invariant $\sigma$-algebra of $T$. 
Suppose that $(z_t)$ is a sample continuous ergodic stochastic process with values in $E$, stationary on the space of paths $C([0,1]; E)$.
Denote by $\mu$ its one time probability distribution.  Then for any real valued function $f\in L^1(\mu)$,
$$\f 1 t \int_0^t f(z_s)ds\to \int f(z) \mu(dz).$$
This is simply Birkhoff's theorem applied to the shift operator  and to the function 
$F(\omega)=\int_0^1 f(z_s(\omega)) ds$.
If $z_t$ is not stationary, but a Markov process with initial value a point, conditions are needed to
ensure the convergence of the Markov process to  equilibrium with sufficient speed. 
\bigskip

We  explain below stochastic averaging for a random field whose randomness is introduced by a fast diffusion.
Let 
$(x_t^\epsilon, y_t^\epsilon)$ be solution to the SDE on $\R^{m_1}\times \R^{m_2}$:
$$dx_t^\epsilon=\sum_{k=1}^{m_1}\sigma_k(x_t^\epsilon, y_t^\epsilon)\,dB_t^k+b(x_t^\epsilon, y_t^\epsilon)\,dt, \quad dy_t^\epsilon=\f 1{\sqrt \epsilon} \sum_{k=1}^{m_2} \theta_k(x_t^\epsilon,  y_t^\epsilon) dW_t^k+\f 1 {\epsilon}b(x_t^\epsilon,y_t^\epsilon)\,dt.$$
with initial values $x_0^\epsilon=x_0$, and $ y_0^\epsilon=y_0$.  Here the stochastic integrations  are  It\^o integrals.
A sample averaging theorem is as following.
Let $z_t^x$ denote the solution to the
SDE  $$dz_t^x=\sum_{k=1}^{m_2}\theta_k(x, z_t^x) dW_t^k+b(x,z_t^x)\,dt$$
with initial value $z_0^x$.
Suppose that the  coefficients are globally Lipschitz continuous and bounded. Suppose that  $\sup_{t\in [0,T]}\sup_{\epsilon \in (0, 1]} \EE|y_t^\epsilon|^2$ and $\sup_x\sup_{t\in [0,T]} \EE |z_t^x|^2$  are finite.
Also suppose that  there exist functions  $\bar a_{i,j}$ and $\bar b$ such that 
\begin{equation}\label{ergodic-assumption}
\begin{aligned} 
\left|\f 1 t \EE \int_0^t b(x, z_s^x)ds-\bar b(x)\right|\le C(t)(|x|^2+|z|^2+1),\\
\left|\f 1 t \EE \int_0^t \sum_{k}\sigma_k^i\sigma_k^j(x, z_s^x)ds-\bar a_{i,j}(x)\right|\le C(t)(|x|^2+|z|^2+1).
\end{aligned}
\end{equation}  
Then the stochastic processes $x_t^\epsilon$ converge weakly to a Markov process with generator  $\f 12  \bar a_{i,j}\f {\partial^2} {\partial x_i\partial x_j}+ \bar b_k\f {\partial^2} {\partial x_k} $, see 
  \cite{Stratonovich61,Stratonovich63,hasminskii68, Borodin77,Veretennikov}.
   See also \cite{Kohler-Papanicolaou74, Kurtz70, Kipnis-Varadhan, Helland,Papanicolaou-Stroock-Varadhan73}. See \cite{Hairer-Pardoux,Barret-vonRenesse,Fu-Liu, Fu-Duan, Dolgopyat-Kaloshin-Koralov,Liverani-Olla,Ruffino, E, Borodin-Freidlin,  Hairer-Pavliotis, Catellier-Gubinelli, Kelly-Melbourne} for a range of  more recent related work. We also refer to the following books
   \cite{Komorowski-Landim-Olla,Pavliotis-Stuart08,Skorohod-Hoppensteadt-Habib}

 Averaging of stochastic differential equations on manifolds has been studied in the following articles \cite{KiferBook88}, \cite{Li-averaging}, \cite{Li-OM-1},  and \cite{Gonzales-Gargate-Ruffino}. In these studies either one restricts to  local coordinates, or has a set of convenient coordinates, or one works directly with local coordinates.
We will be using a global approach.

We will first deduce a locally uniform Birkhoff's ergodic theorem for $\L_x$, then prove an averaging theorem for  (\ref{slow-fast-sdes}). Finally we study a number of examples of singular perturbation problems. 

The main assumptions on $\L_x$ is a H\"ormander's (bracket) condition.
\begin{definition}\label{def-hormander}
Let $X_0, X_1, \dots, X_k$ be smooth vector fields.
\begin{enumerate}
\item The differential  operator $\sum_{k=1}^m (X_i)^2+X_0$  is said to satisfy {\it H\"ormander's condition}  if  $\{X_k, k=0, 1, \dots, m\}$ and their iterated Lie brackets generate the tangent space at each point.
\item The differential  operator $\sum_{k=1}^m (X_i)^2+X_0$, is said to satisfy {\it strong H\"ormander's condition}  if  $\{X_k, k=1, \dots, m\}$ and their iterated Lie brackets generate the tangent space at each point.

\end{enumerate}
\end{definition}
%This means that $\{Y_i(x, \cdot), i=0, 1, \dots, m_2\}$ and their iterated Lie brackets generate
%the tangent spaces of $G$  at each point.  If $\{Y_i(x, \cdot), i=1, \dots, m_2\}$ and their iterated Lie brackets, without using the drift vector field $Y_0(x, \cdot)$,  generate the tangent spaces at each point, we say that $\L_x$ satisfies the strong H\"ormander's condition.

{\it Outline of the paper.} In \S\ref{Birkhoff} we study the regularity of invariant probability measures $\mu^x$ of $\L_x$ with respect to the parameter $x$  and prove the local uniform law of large numbers with rate. We may assume that each $\L_x$ satisfies H\"ormander's condition.
What we really need is that  $\L_X$ is a family of Fredholm operators satisfying the sub-elliptic estimates and with zero Fredholm index.
In \S\ref{estimates} we give estimates for SDEs on manifolds. It is worth noticing that we do not assume that the transition probabilities have densities.  We use an approximating family of distance functions to overcome the problem that the distance function is not smooth. These estimates
lead easily  to the tightness of the slow variables. 
In \S\ref{proof} we prove the convergence of the slow variables, for which we first prove a theorem on time averaging of path integrals of the slow variables. This is proved under a law of large numbers with any uniform rate.
In \S\ref{examples} we study some examples of singular perturbation problems. Finally,  we pose a number of open questions,  one of which is presented in the next section, the others are presented in \S\ref{examples}.

 \subsection{Description of results.}
\label{results}
The following  law of large numbers with a locally uniform rate  is proved is section~\ref{Birkhoff}.

\begin{theorem}
[Locally Uniform  Law of Large Numbers]
\label{lln}
Let $G$ be a compact manifold. Suppose that $Y_i$ are bounded, $C^\infty$ with bounded derivatives. Suppose that  each
$$\L_x=\f 12\sum_{i=1}^m  Y^2_i(x, \cdot) +Y_0(x, \cdot)$$ satisfies H\"ormander's condition (Def. \ref{def-hormander}),
and has a unique invariant probability measure  $\mu_x$. Then the following statements hold for $\mu_x$.
\begin{enumerate}
\item[(a)]  
  $x\mapsto \mu_x$ is locally Lipschitz continuous in the total variation norm.
\item[(b)] 
  For every  $s>1+\f{\dim(G)}{2}$ there exists a positive constant $C(x)$, depending continuously in $x$,  such that for every smooth  function $f:G \to \R$, 
\begin{equation}
\label{llln-1}
\left| {1\over T} \int_t^{t+T} f(z_r^x) \;dr- \int_G f(y) \mu_x(dy) \right|_{L_2(\Omega)} 
\le C(x)\|f\|_s\f 1{\sqrt T},
\end{equation}
where  $z_r$ denotes an $\L_x$-diffusion.

\end{enumerate}

\end{theorem}
\begin{remark}
Let $P^x(t, y, \cdot)$ denote the transition probability of $\L_x$.
 Suppose $\L_x$ satisfies Doeblin's condition. Then $\L_x$ has a unique invariant probability measure.
This holds in particular if $\L_x$ satisfies the strong H\"ormander's condition and $G$ is compact.
 The uniqueness follows from the fact that it has a smooth strictly positive density. 
 (  H\"ormander's condition ensures that any invariant measure has a smooth  kernel and the  kernel of its $L_2$ adjoint $\L^*$  contains a non-negative function. The density  is however not necessarily positive. )  Suppose that each $\L_x$ satisfies the strong H\"ormander's condition (c.f. Def. \ref{def-hormander}) and $G$ is compact. It is well know that the transition probability measures $P^x(t, y_0, \cdot)$, with any initial value $y_0$,  converges to  the unique invariant probability measure $\mu^x$ with an exponential rate which we denote by $C(x)e^{\gamma(x)t}$.  If $x$ takes values also in a compact space $N$, the  exponential rate  and the  constant in front of  the exponential rate can be taken to be independent of $x$. When $N$ is non-compact, we obviously need to make further assumptions on $\L_x$ for a uniform estimate. 
 \end{remark}

 Set
$$\begin{aligned}
\tilde X_0( \cdot, y)&=\f 12 \sum_{i=1}^{m_1}\nabla X_i(X_i)(y, \cdot)+X_0(y, \cdot), \\
\tilde Y_0(x, \cdot) &=\f 12 \sum_{i=1}^{m_2} \nabla Y_i(Y_i)(x, \cdot)+Y_0(x, \cdot).
\end{aligned}
$$
 Let $O$ be a reference point in $N$ and $\rho$ is the Riemannian distance from $O$.

\begin{assumption}
[Assumptions on $X_i$ and $N$]
\label{assumeX}
Suppose that $\tilde X_0$  and $X_i$ are $C^1$, where $i=1,\dots, m$. Suppose that {\bf one} of the following  two statements holds.
\begin{enumerate}
\item [(i)]
The sectional curvature of $N$ is bounded. There exists a constant  $K$ such that  $$\sum_{i=1}^m|X_i(x,y) |^2\le K(1+ \rho(x)), \quad\left  |X_0(x,y) \right| \le  K(1+ \rho(x)), \quad \forall x\in N, \forall y\in G.$$
 \item [(ii)] Suppose that  the square of the distance function on $N$ is smooth. Suppose that $$\f 12 \sum_{i=1}^m \nabla d \rho^2(X_i(\cdot,y), X_i(\cdot, y))+ d\rho^2 (\tilde X_0(\cdot,y ))\le  K+ K\rho^2(\cdot), \quad \forall y\in G.$$
\end{enumerate}
\end{assumption}

\begin{assumption}[Assumptions on $Y_i$ and $G$]
\label{assumeY}
We suppose that $G$ has bounded sectional curvature.
 Suppose that $\tilde Y_0$ and $Y_j$ are $C^2$ and bounded with  bounded first order derivatives. 
\end{assumption}

The following is extracted from  Theorem \ref{main-2}.

\begin{theorem} [Averaging Theorem]
\label{theorem2}
Suppose that there exists a family of  invariant probability measure  $\mu_x$ on $G$ that satisfies the conclusions of Theorem \ref{lln}.
Suppose the assumptions on $X_i$, $N$,   $Y_i$ and $G$ hold (Assumptions \ref{assumeX} and \ref{assumeY}).  Then as $\epsilon\to 0$,  the stochastic processes  $x_t^\epsilon$  converges weakly on $C([0,T],N)$ to a Markov process with generator~$\bar \L$.
\end{theorem} 

\begin{remark}
\begin{enumerate}
\item [(i)]
 If $f$ is a smooth function on $N$ with compact support then 
\begin{equation}
\label{limit}
\bar \L  f(x)=\int_G \left(\f 12 \sum_{i=1}^{m_1} X_i^2(\cdot, y) f+ X_0 (\cdot, y) f\right) (x)\, \mu_x(dy).
\end{equation}
See the Appendix in \S\ref{proof} for a sum of squares of vector fields decomposition of~$\bar \L$.
\item[(ii)]   Under Assumptions \ref{assumeX}  there exists a unique global solution $x_t^\epsilon$ for each initial value $(x,y)$. We also have uniform estimates on the distance  $\rho(x_s^\epsilon, x_t^\epsilon)$ which leads to the conclusion that the family $\{x_\cdot^\epsilon, \epsilon>0\}$ is tight. Also we may conclude that the moments of the solutions are bounded uniformly on any compact time interval and in $\epsilon$ for $\epsilon\in(0, 1]$. Such estimates  are given in \S\ref{estimates}. 

\item [(iii)] Under Assumptions \ref{assumeX}- \ref{assumeY}   we may approximate  the fast motion, on sub-intervals $[t_i, t_{i+1}]$, by freezing the slow variables and obtain a family of Markov processes with generator $\L_x$. 
The size of the sub-intervals must be of size $o(\epsilon)$  for the error of the approximation to converge to zero as $\epsilon\to 0$, and large on the scale of $\f 1\epsilon$ for the ergodic average to take effect.  
\end{enumerate}\end{remark}

\begin{problem}
Suppose that $\L_x$ satisfies H\"ormander's condition.  Then the kernel of $\L_x^*$ is finite dimensional. Without assuming the uniqueness of the invariant probability measures, it is possible
to define a projection to the kernel of $\L_x$, by pairing up a basis $\{u_i(x)\}$  of $\ker (\L_x)$ with a dual basis $\pi^i(x)$ of $\ker(\L_x^*)$ and this leads to a family of projection operators $\Pi(x)$.
To obtain a locally uniform version of this, we should consider the continuity of $\Pi$ with respect to $x$. 
Let us consider the simple case of  a family of Fredhom operators $T(x)$ from a Hilbert space $E$ to a Hilbert space $F$. It is well known that the dimension of their kernels may not be a continuous function of $x$, but the Fredholm index of $T(x)$ is a continuous function if $x$ in the space of bounded linear operators \cite{Atiyah}. 
See also \cite {Erp-index,Erp-index-foliated} for non-elliptic operators. Given that the projection $\pi(x)$ involves both the kernel and the co-kernel, it is reasonable to expect that $\Pi(x)$ is continuous in $x$. The question is whether this is true and more importantly whether in this situation there is a local uniform Law of large numbers.
\end{problem}

\section{ Examples}
 \label{examples}

 We describe  some motivating examples, the first being  dynamical  descriptions for Brownian motions,
 the second being the convergence of metric spaces. The overarching question concerning the second is: given a family of metric spaces converging to another in measured Gromov Hausdorff topology, can we give a good dynamical description for their convergence?
 What can one say about the associated singular operators? These will considered in terms of stochastic dynamics. See \cite{Ikeda-Ogura,Ogura-Taniguchi96} and \cite{Li-OM-1} concerning collapsing of Riemannian manifolds. The third example is a model on the principal bundle. 
 These singular perturbation models were introduced in \cite{Li-averaging, Li-geodesic, Li-homogeneous, Li-OM-1},  where the perturbations were chosen carefully for diffusion creation. The reduced systems are random ODEs for which a set of limit theorems are available,  and the perturbations are chosen so  that one could bypass the stochastic averaging procedure and work directly on the faster scale for diffusion creation $[0, \f 1 {\epsilon}]$. Theorem \ref{main-2} allows us to revisit these models to include more general perturbations, in which the effective limits on $[0, 1]$ are not trivial. It also highlights from a different angle the choice of the perturbation vector in the models which we explain below.

\subsection{A dynamical description for Brownian motions}\label{dynamical}

In 1905, Einstein, while working on his atomic theory,  proposed the diffusion model for  describing the density of the probability for finding a particle at time $t$ in a position $x$.
A similar model was proposed by Smoluchowski with a force field. Some years later Langevin (1908) and Ornstein-Uhlenbeck (1930) \cite{ Ornstein-Uhlenbeck}  proposed a dynamical model for Brownian motion for time  larger than the relaxation time $\f 1 \beta$: 
\begin{equation*}
\left\{{\begin{split}
{\dot x(t)}&{=v(t)}\\
{\dot v(t)}&{=-\beta v(t) dt+\sqrt D \beta dB_t+\beta b(x(t))dt}
\end{split}}\right.
\end{equation*}
where $B_t$ a one dimensional Brownian motion and $b$ a vector field.
This equation is stated for $\R$ with $\beta, D$ constants and  was studied by Kramers \cite{Kramers}  and Nelson \cite{Nelson}.  The model is on the real line, there exists only one direction for the velocity field. The  magnitude of $v(t)$ together with the sign changes rapidly.

The second order differential equations for unit speed geodesics, on a manifold $M$,  are equivalent to  first order ODEs on the space of orthonormal frames of $M$, this space will be denoted by $OM$. Suppose that we rotate the direction of the geodesic uniformly, according to the probability distribution of 
a Brownian motion on $SO(n)$, while keeping its magnitude fixed to be $1$,  and suppose that the rotation is at the scale of $\f 1\epsilon$ then the projections to $M$ of the solutions of the equations on $OM$ converge to a fixed point as $\epsilon \to 0$. But if we further tune up the speed of the rotation, these motions converge to a scaled Brownian motion, whose scale is given by an eigenvalue of the fast motion on $SO(n)$.  See \cite{Li-geodesic}. An extension to manifolds was first studied 
 \cite{Dowell} followed  by \cite{Bismut-Lebeau}. That in   \cite{Dowell} is different from that in {Li-geodesic}, which is also followed up in \cite{Angst-Bailleul-Tardif} where the authors iremoved the geometric curvature restrictions in \cite{Li-geodesic}. See also \cite{Birrell-Hottovy-Volpe-Wehr} for a local coordinate approach and  more recently \cite{Duong-Lamacz-Peletier-Sarma}.
Assume the dimension of $M$ is greater than $1$. The equations, \cite{Li-geodesic}, describing this are as following:
\begin{equation}\label{ou-1}\left\{
\begin{split}du_t^\epsilon&= H_{u_t^\epsilon}(e_0)\,dt
+{1\over \sqrt \epsilon} \sum_{k=1}^NA_k^*(u_t^\epsilon)\circ dW_t^k+ A_0^*
(u_t^\epsilon) \,dt,
 \\ u_0^\epsilon&=u_0.
\end{split}\right.
\end{equation}
where $\{A_1,\dots, A_N\}$ is an o.n.b. of $\so(n)$, and $ A_0\in \so(n)$. The star sign denotes the corresponding vertical fundamental vector fields and $H(u)(e_0)$ is the horizontal vector field corresponding to a unit vector $e_0$ in $\R^n$. This following theorem is taken from \cite{Li-geodesic}.

{\bf Theorem 2A.}    The position part of $u_{\f t \epsilon}^\epsilon$,  which we denote by $(x_{t\over \epsilon}^\epsilon) $, converges to a Brownian motion on $M$ with generator ${4\over n(n-1)}\Delta$.
Furthermore  the parallel translations along these smooth paths of $(x_{t\over \epsilon}^\epsilon)$ converge to stochastic parallel translations along the H\"order continuous sample paths of the effective scaled Brownian motion. 

 The conservation law in this case is the projection $\pi$, taking a frame to its base point, using which we obtain the following reduced system of slow-fast SDEs:
  \begin{equation*}\left\{\begin{aligned}
& \f {d} {dt} \tilde x_t^\epsilon=H_{\tilde x_t^\epsilon} (g_{\f t \epsilon} e_0),
 \quad  \tilde x_0^\epsilon=u_0,\\
&dg_t=  \sum_{k=1}^m{g_t} A_k \circ dw_t^k+g_t A_0\, dt, \quad g_0=Id. 
\end{aligned}\right.
\end{equation*}
The slow variable does not have a stochastic part,  the averaging equation is given by the average vector field
$\int_{SO(n)} H(ug)(e) dg$, where $dg$ is the Haar measure, and vanishes. Hence we may observe the
slow variable on a faster scale and consider $x^\epsilon_{\f t \epsilon} $. 

In section \ref{dynamical2} we use the general results obtained later  to study two generalised models.

\subsection{Collapsing of manifolds}
Our  overarching question is how the stochastic dynamics  describe the convergence of metric spaces. Let us consider a simple example: $SU(2)$ which can be identified with the sphere $S^3$. The Lie algebra of $SU(2)$ is given by the Pauli matrices
    $$X_1=\left(\begin{matrix} i &0\\0&-i
\end{matrix}\right), \quad X_2=\left(\begin{matrix} 0 &1\\-1&0
\end{matrix}\right),
 \quad X_3=\left(\begin{matrix} 0 &i\\i&0
\end{matrix}\right).$$
By declaring $\{\f 1 {\sqrt \epsilon} X_1, X_2, X_3\}$  an orthonormal frame we define  Berger's metrics $g^\epsilon$. Thus $(S^3, g^\epsilon)$ converges to $S^2$. 
  They are the first known family of  manifolds which collapse to a lower dimensional one,
 while keeping the sectional curvatures uniformly bounded  (J. Cheeger).
 Then all the operators in the sum  
$$\Delta_{S^3}^{\epsilon} =\f 1 \epsilon (X_1)^2+(X_2)^2 +(X_3)^2=\f 1 \epsilon \Delta_{S^1}+\Delta_H$$
commute, the eigenvalues satisfy the relation $\lambda_3(\Delta_{S^3}^{\epsilon})= \f 1 \epsilon \lambda_1(\Delta_{S^1})+\lambda_2(\Delta_H)$. The non-zero eigenvalues of $\Delta_{S^1}$ flies away and the eigenfunctions of $\lambda_1=0$ are function on  the sphere $S^2(\f 12)$ of radius $\f 12$, the convergence of the spectrum of $\Delta_{S^3}^\epsilon$ follows.  See \cite{Tanno-79},  \cite{Berard-Bergery-Bourguignon} \cite{Urakawa86} for discussions on the spectrum of Laplacians on spheres, on homogeneous Riemannian manifolds and on Riemannian submersions with totally geodesic fibres. 

We  study   $${\L^\epsilon:=\f 1 \epsilon \Delta_{S^1}+Y_0}$$ 
in which $\Delta_{S^1}$ and $Y_0$ do not commute. Take for example, 
 $Y_0 =aX_2+b X_3$ where $|Y_0|=1$.
Let $\pi(z,w)=\f 1 2(|w|^2-|z|^2, z \bar w)$ be the Hopf map. 
Let  $u_t^\epsilon$ be an $\L^\epsilon$-diffusion with the initial value $u_0$. Then  $\pi(u_{\f t \epsilon} ^\epsilon)$  converges to a BM on $S^2(\f 12)$, 
{scaled by $\f  12$}. See \cite{Li-homogeneous}. See also  \cite{Ogura-Taniguchi96} for related studies. It is perhaps interesting  to observe that $\L^\epsilon$ satisfies H\"ormander's condition for any $Y_0\not =0$. Later we see that this fact is not  an essential feature of the problem.
The model  on $S^3$ in \cite{Li-OM-1}  is a variation of this one.

\subsection{Inhomogeneous scaling of Riemannian metrics}\label{scaling}
 If a manifold is given a family of Riemannian metrics depending on a small parameter $\epsilon>0$, the  Laplacian operators $\Delta^\epsilon$
 is a family of singularly perturbed operators. We might ask the question whether their spectra converge.  More generally let us consider a family of second order differential operators $\L^\epsilon=\f 1 \epsilon \L_0+\L_1$, each
 in the form of a finite sum of squares of smooth vector fields with possibly a first order term.
As $\epsilon \to 0$, the  corresponding Markov process does not converge in general.
In the spirit of Noether's theorem, to see a convergent slow component we expect to see some symmetries for the system $\L_0$.
On the other hand, by a theorem of  S. B. Myers and N. E. Steenrod \cite{Myers-Steenrod},   the set of all isometries of a Riemannian manifold $M$ is a Lie
group under composition of maps, and  furthermore the isotropy subgroup $\Iso_o(M)$ is compact.  See also S. Kobayashi and K. Nomizu \cite{Kobayashi-NomizuI}.
We are led to study homogeneous manifolds $G/H$, where $G$ is a smooth Lie group and $H$ is a compact sub-group of  $G$. We denote by $\g$ and $\h$ their respective Lie algebras.

Let  $\g$ be endowed an $\Ad(H)$-invariant inner product and take $\m=\h^\perp$. Then
$G/H$ is a reductive homogeneous manifold, in the sense of Nomizu, by which we mean $\Ad(H)(\m)\subset \m$. This is a different  from the concept of a reductive Lie group, where
 the adjoint representation of  the Lie group $G$ is completely reducible.  (Bismut studied  a natural deformation of the standard Laplacian on a compact Lie group $G$ into a hypoelliptic operator on $TG$ see \cite{Bismut-Lie-group}.)
We assume that the real Lie group $G$ is smooth, connected, not necessarily compact,  of dimension $n$ and $H$ a closed connected proper subgroup of dimension at least one. 
 We identify elements of the Lie algebra with left invariant vector fields.
 
We generate a family of Riemannian  metrics on $G$ by scaling the $\h$ directions by $\epsilon$.  
Let $\{A_1, \dots,  A_p, Y_{p+1}, \dots, X_N\}$ be an orthonormal basis of $\g$ for an inner product extending an orthonormal basis $\{A_1, \dots, A_p\}$ of $\h$ with the remaining vectors  from $\m$.
By declaring $$\left\{{1\over \sqrt \epsilon} A_1, \dots, {1\over \sqrt \epsilon} A_p, Y_{p+1}, \dots Y_N\right\}$$ 
   an orthonormal frame, we obtain a family of  left invariant Riemannian metrics.
   Let us consider the following second order differential operator,  related to the re-scaled metric:
   $$\L^\epsilon=\f 1 {2\epsilon} \sum_{k=1}^{m_2} (A_k)^2+\f 1 \epsilon A_0+Y_0,$$
   where $A_k\subset \h$ and $Y_0 \in \m$ is a unit vector. 
   This leads to the following family of equations, where $\epsilon \in (0,1]$,
$$dg_t^\epsilon ={1\over \sqrt \epsilon} \sum_{k=1}^{m_2}   A_k(g_t^\epsilon) \circ dB_t^k +{1\over \epsilon} A_0(g_0^\epsilon)dt
+Y_0(g_t^\epsilon) dt, \quad g_0^\epsilon =g_0.$$ 
These SDEs  belong to the following family of equations
\begin{equation*}\label{1-2}
dg_t =\sum_{k=1}^{m_2}  \gamma A_k(g_t) \circ dB_t^k +\gamma \, A_0(g_0)\,dt +
\delta Y_0(g_t) dt.
\end{equation*}
The solutions of the latter family of equations,  with parameters $\gamma$ and $\delta$ real numbers,  interpolate between translates of a one parameter subgroups of $G$ and diffusions on~$H$. 
Our  study of $\L^\epsilon$ is related to the concept o `taking the adiabatic limit' \cite{Berline-Getzler-Vergne, Mazzeo-Melrose}.

Let $(g_t^\epsilon)$ be a Markov processes with Markov generator $\L^\epsilon$, and set $x_t^\epsilon=\pi(g_t^\epsilon)$ where $\pi$ is the map taking an element of $G$ to the coset $gH$.
 Then $\L^\epsilon=\f 1 \epsilon \L_0+Y_0$ where $\L_0=\f 1 {2} \sum_{k=1}^{m_2} (A_k)^2+A_0$.
   We will assume that $\{A_k\}\subset \h$ are bracket generating.
   Scaled by $1/\epsilon$, the Markov generator of  $(g_{t\over \epsilon}^\epsilon)$ is precisely $\f 1\epsilon \L^{\epsilon}$.

The operators $\L^\epsilon$ are not necessarily hypo-elliptic in $G$, and they will not expected to converge in the standard sense.  
Our first task is to understand the nature of the perturbation and to extract from them a family of first order random differential operators, $ \tilde \L^\epsilon$,  which converge and which have the same orbits as $ \L^\epsilon$,  the `slow motions'.  The reduced operators, $\f 1 \epsilon\tilde\L^\epsilon$, describe the motion of the orbits  under  `perturbation'.

Their effective limit is either a one parameter sub-groups of $G$ in which case our study terminate, or a fixed point in which case we study the fluctuation dynamics on the time scale $[0, \f 1 \epsilon]$. On the Riemmanian homogeneous manifold, if $G$ is compact,  the effect limit on $G$ is a geodesic at level one and a fixed point at level two.
 On the scale of $[0, \f 1 \epsilon]$ we would  consider $\f 1 {\epsilon} \L_0$ as perturbation. It is counter intuitive to consider the dominate part as the perturbation. But the perturbation, although very large in magnitude, is fast oscillating. The large oscillating motion get averaged out, leaving an effective motion corresponding to a second order differential operator on~$G$.

 This problems breaks into three parts:  separate the slow and the fast variable, which depends on the principal bundle structure of the homogeneous space, and determine the natural scales;
 the convergence of the solutions of the reduced equations which is a family of random ODEs;
 finally the buck of the interesting study is to  determine the effective limit, answering the question whether it solves an autonomous equation.
 
 It is fairly easy to see that $x_t^\epsilon$ moves relatively slowly. The speed at which $x_t^\epsilon$ crosses $M$ is expected to depend on the specific vector $Y_0$, however in the case of $\{A_1, \dots, A_p\}$ is an o.n.b. of $\h$ and $A_0=0$,
they depend only on the $\Ad(H)$-invariant component of $Y_0$.

  The separation of slow and fast variables are achieved by first projecting the
   motion down to $G/H$ and then horizontally lift the paths back (a non-Markovian procedure), exposing the action in the fibre directions.  The horizontal process thus obtained is the `slow part' of   $g_t^\epsilon$  and will be denoted by $u_t^\epsilon$.
 It is easy  to see that the reduced dynamic  is given by
$$\dot u_t^\epsilon =\Ad(h_{\f t \epsilon})(Y_0)(u_t^\epsilon).$$
where $h_t$ has generator $\f 12 \sum (A_i)^2+A_0$. 

 If $\{A_0, A_1, \dots, A_m\}$ generates
the vector space $\h$, the differential operator $\f 12 \sum (A_i)^2+A_0$ satisfies H\"ormander's condition in which case the invariant probability measure is the normalised Haar measure.
Then $u_t^\epsilon$ converges to the solution of the ODE:
$$\f d {dt}\bar u_t = \int_H \Ad(h)(Y_0)dh.$$
Let us take an $\Ad(H)$ invariant decomposition of $\m$, $\m=\m_0+\m_1$ where $\m_0$ is the vector space of invariant vectors and $\m'$ is its orthogonal complement. Then 
$$\int_H \Ad(h)(Y_0)dh={Y_0}^{\m_0}$$
where the superscript $\m_0$ denote the $\m_0$ component of $Y$. This means that the dynamics is a fixed point if and only if $Y_0^{\m_0}=0$. 

In \cite{Li-homogeneous} we take $Y_0^{\m_0}=0$ and answered this question by a multi-scale analysis and studied directly the question concerning $Y_0\in \tilde \m$, without having to go through stochastic averaging.
Theorem 1 makes this procedure easier to understand. 
Then we consider the dynamics on $[0, \f 1 \epsilon]$. The reduced first order random differential operators give rise to second order differential operators by the action of the Lie bracket.

\subsection{Perturbed dynamical systems  on Principal bundles}\label{averaging}
In the examples  described earlier, we have a perturbed dynamical system on a manifold $P$. On $P$ there is an action by a Lie group $G$, and the  projection to $M=P/G$ is a conservation law.
  We then study the convergence of the slow motion, the projection to $M$, and their horizontal lifts.  
  More precisely we have a principal bundle with fibre the Lie group $G$.   To describe these motions we consider the kernels of the differential of the projection $\pi$: they are  called the vertical tangent spaces and will be denoted by $VT_uP$. Any vector field 
 taking values in the vertical tangent space is called a vertical vector field, the Lie-bracket of
 any two vertical vector fields is vertical. A smooth choice of the complements of the vertical spaces,
   that are right invariant, determines a connection. These complements are called the horizontal spaces. The ensemble is denoted by $HT_uP$ and called the horizontal bundle.    From now on we assume that we have chosen such a horizontal space. A vector field taking values in the horizon tangent spaces is said to be a horizontal vector field.
Right invariant horizontal vector fields are specially interesting, they  are precisely the horizontal  lifts of  vector fields on $M$.

Let $\pi: P\to M$ denote the canonical projection taking an element of the total space $P$ to the corresponding element of the base manifold. Also  let $R_g: P\to P$ denote  the  right action by $g$, for simplicity we also write
$ug$, where $u\in P$, for $R_gu$.
A  connection on a principal bundle $P$ is a splitting of the  tangent bundle
$T_uP=HT_uP+VT_uP$ where $VT_uP$ is the kernel of the differential of t$\pi$. Let $\g$ denote the  Lie algebra of $G$. For any $A\in \g$ we define
$$A^*(u) =\lim_{t\to 0} R_{ \exp(tA)}u.$$
The splitting mentioned earlier  is in one to one correspondence  with  a connection 1-form, by which we mean a map
$\varpi: T_uP\to \g $ with the following properties:
$$(R_g)^*\varpi=\ad(g^{-1})\varpi, \quad \varpi(A^*)\equiv A.$$
This splitting also determines a horizontal lifting map $\h_u$ at $u\in P$ and a family of 
horizontal vector fields $H_i$. If $\{e_1, \dots e_n\}$ is an orthonormal basis of $\R^n$, where $n=\dim(M)$, we set $H_i(u)=\h_u(ue_i)$. If $\{A_1, \dots , A_N\}$ is an orthonormal basis of the Lie algebra
$\g$,  then at every point $u$, $\{H_1(u), \dots, H_n(u), A_1^*(u), \dots, A_N^*(u)\}$ is a basis of $T_uP$. We give $P$ the Riemannian metric so that the basis is orthonormal.

 Any stochastic differential equation (SDE) on $P$ are of the following form, where $\beta$ and $\gamma$ are two real positive  numbers and $\sigma_j^k$ and $\theta_j^k$ are $BC^3$ functions on $P$.

 $$\begin{aligned} du_t=&\beta \sum_{k=1}^{m_1} \left(\sum_{i=1}^n \sigma_k^i(u_t) H_i(u_t)\right)\circ dB_t^k+
\beta^2 \sum_{i=1}^n \sigma_0^i(u_t) H_i(u_t)dt\\
& +\gamma^2\sum_{k=1}^{m_2}\left( \sum_{j=1}^N \theta_k^j(u_t) A_j(u_t)\right)\circ dW_t^k+
\gamma \sum_{j=1}^N\theta_0^j(u_t) A_j(u_t)dt.\end{aligned}$$
Set $X_k= \sum_{i=1}^n \sigma_k^i \, H_i,$ and $ Y_k= \sum_{j=1}^N \theta _k^j\, A_j$. Then the equation is of the form
  $$ du_t=\beta\, \sum_{k=1}^{m_1} X_k(u_t)\circ dB_t^k+
\beta^2\; X_0(u_t)\;dt +\gamma \;\sum_{k=1}^{m_2}Y_k \circ dW_t^k+
\gamma^2\; Y_0(u_t)dt.$$
The solutions are Markov processes with Markov generator
$$\beta^2\left( \sum_{k=1}^{m_1} (X_k)^2+X_0 \right) +\gamma^2\left( \sum_{k=1}^{m_2} (Y_k)^2+Y_0 \right). 
$$
We observe that  the projection of the second factor vanishes, so if $\beta=0$, then  $\pi(u_t)=\pi(u_0)$ and $\pi$ is a conservation law. The equation with
small $\beta$ is a stochastic dynamic whose orbits deviate slightly  from that of the initial value $u_0$.
If on the other hand,
$X_i$ are vector fields invariant under the action of the group, and $\gamma=0$ then the projection $\pi(u_t)$ is an autonomous SDE on the manifold $M$.

Let us take $\beta=1$ and $\gamma=\f 1 {\sqrt \epsilon}$.
$$\left\{\begin{aligned} du_t^\epsilon&=\sum_{k=1}^{m_1} \left(\sum_{i=1}^n \sigma_k^i(u_t^\epsilon ) H_i(u_t^\epsilon)\right)\circ dB_t^k+
 \sum_{i=1}^n \sigma_0^i(u_t^\epsilon ) H_i(u_t^\epsilon)\;dt\\
 & +\f 1{ \sqrt{ \epsilon}}
\sum_{k=1}^{m_2}\left( \sum_{j=1}^N \theta_k^j(u_t^\epsilon) A_j^*(u_t^\epsilon)\right)\circ dW_t^k+
\f 1 {\epsilon} \sum_{j=1}^N \theta_0^j(u_t^\epsilon ) A_j^*(u_t^\epsilon)dt,\\
u_0^\epsilon&=u_0.\end{aligned}\right.$$
We proceed to compute the equations for the slow and for the fast variables.
  Let $x_t^\epsilon=\pi(u_t^\epsilon)$. Then $x_t^\epsilon$ has a horizontal lift, see e.g. \cite{Emery, Arnaudon-homogeneous, Elworthy-book}. See also \cite{Elworthy-LeJan-Li-book} and \cite{Elworthy-LeJan-Li-book-2}.
 Let $TR_g$ denote the differential of $R_g$.  For $k=0,1, \dots, m_1$, set $$X_k(ug)=\sum_{i=1}^p \sigma_k^i(ug )\, TR_{g^{-1}}H_i(ug).$$
 
 Below we deduce an equation for $x_t^\epsilon$ which is typically not autonomous.
  \begin{lem}
  The horizontal lift processes satisfy the following system of slow-fast SDE's:
  \begin{equation}\begin{aligned}
 d\tilde x_t^\epsilon&=\sum_{k=1}^{m_1} X_k\left(\tilde x_t^\epsilon g_t^\epsilon \right)\circ dB_t^k+
X_0\left(\tilde x_t^\epsilon g_t^\epsilon \right)\;dt, \quad \tilde x_0^\epsilon=g_0\\
dg_t^\epsilon
&=\f 1{ \sqrt{ \epsilon}}
\sum_{k=1}^{m_2}\left( \sum_{j=1}^N \theta_k^j(\tilde x_t^\epsilon g_t^\epsilon) A_j^*(g_t^\epsilon)\right)\circ dW_t^k+
\f 1 {\epsilon} \sum_{j=1}^N \theta_0^j(\tilde x_t^\epsilon g_t^\epsilon) A_j^*(g_t^\epsilon)\,dt, \; g_0^\epsilon=id.
\end{aligned}
\end{equation}
  \end{lem}
  \begin{proof}
   Since $\tilde x_t^\epsilon$ and $u_t^\epsilon$ belong to the same fibre we may define
 $g_t^\epsilon \in G$  by $u_t^\epsilon=\tilde x_t^\epsilon g^\epsilon_t$. If $a_t$ is a $C^1$ curve in the lie group $G$ 
 $$ {d\over dt}|_{t}ua_t={d\over dr}_{|_{r=0}} ua_t a_t^{-1}a_{r+t} =(a_t^{-1}\dot a_t)^*(ua_t).$$
It follows that
$$du_t^\epsilon=TR_{g_t^\epsilon} d\tilde x_t^\epsilon+ (TL_{(g_t^\epsilon)^{-1}}dg_t^\epsilon)^*(u_t^\epsilon).$$
Since right translations of horizontal vectors are horizontal, 
$$TL_{(g_t^\epsilon)^{-1}}dg_t^\epsilon=\varpi(du_t^\epsilon) 
=\f 1{ \sqrt{ \epsilon}}
\sum_{k=1}^{m_1}\left( \sum_{j=1}^N \theta_k^j(u_t^\epsilon) A_j\right)\circ dW_t^k+
\f 1 {\epsilon} \sum_{j=1}^N \theta_0^j(u_t^\epsilon ) A_j\,dt$$
Hence, denoting by $A^*$ also the left invariant vector fields on $G$, we have an equation for $g_t^\epsilon$:
$$
dg_t^\epsilon
=\f 1{ \sqrt{ \epsilon}}
\sum_{k=1}^{m_2}\left( \sum_{j=1}^N \theta_k^j(u_t^\epsilon) A_j^*(g_t^\epsilon)\right)\circ dW_t^k+
\f 1 {\epsilon} \sum_{j=1}^N \theta_0^j(u_t^\epsilon ) A_j^*(g_t^\epsilon)\,dt.$$
Since $\pi_*(A_j)=0$ and by the definition of $H_i$ we also have,
$$ dx_t^\epsilon=\sum_{k=1}^{m_1} \left(\sum_{i=1}^p \sigma_k^i(u_t^\epsilon )\, (u_t^\epsilon e_i)\right)\circ dB_t^k+
 \sum_{i=1}^n \sigma_0^i(u_t^\epsilon ) (u_t^\epsilon e_i)\;dt.$$
Its horizontal lift is given by $d\tilde x_t=\h_{\tilde x_t} (\circ dx_t^\epsilon)$ and so we have the following SDE
$$\begin{aligned}
 d\tilde x_t^\epsilon&=\sum_{k=1}^{m_1} \left(\sum_{i=1}^p \sigma_k^i(u_t^\epsilon )\, \h_{\tilde x_t^\epsilon}(u_t^\epsilon e_i)\right)\circ dB_t^k+
 \sum_{i=1}^n \sigma_0^i(u_t^\epsilon )  \h_{\tilde x_t^\epsilon}(u_t^\epsilon e_i)\;dt. \end{aligned}$$
 Since
$\h_u(uge_i)=TR_{g^{-1}} \h_{ug} (uge_i)= TR_{g^{-1}} H_i(ug)$, we may rewrite the above equation in the following more convenient form:
\begin{equation}
 d\tilde x_t^\epsilon=\sum_{k=1}^{m_1} \left(\sum_{i=1}^p \sigma_k^i(\tilde x_t^\epsilon  g_t^\epsilon )\, TR_{g_t^\epsilon}^{-1} 
H_i(\tilde x_t^\epsilon  g_t^\epsilon)\right)\circ dB_t^k+
 \sum_{i=1}^n \sigma_0^i(\tilde x_t^\epsilon  g_t^\epsilon )  TR_{g_t^\epsilon}^{-1} H_i(\tilde x_t^\epsilon  g_t^\epsilon)\;dt.
\end{equation}
Finally we also rewrite the equation for the fast variable in terms of the fast and slow splitting:
\begin{equation}
\label{left-BM}
dg_t^\epsilon
=\f 1{ \sqrt{ \epsilon}}
\sum_{k=1}^{m_2}\left( \sum_{j=1}^N \theta_k^j(\tilde x_t^\epsilon g_t^\epsilon) A_j^*(g_t^\epsilon)\right)\circ dW_t^k+
\f 1 {\epsilon} \sum_{j=1}^N \theta_0^j(\tilde x_t^\epsilon g_t^\epsilon) A_j^*(g_t^\epsilon)\,dt.
\end{equation}
 This completes the proof.
  \end{proof}
  If $\theta_k^j$ are lifts of functions from $M$, i.e. equi-variant functions, then the system of SDEs for $g_t^\epsilon$ do not depend on the slow variables.
 Define
  $$ \L_u f(g)={1\over 2}\sum_{k=1}^{m_1} \left( \theta_k^j(ug) A_j^*(g)\right)^2f(g)
  +\sum_{j=1}^N \theta_0^j(ug) A_j^*(g)f(g).$$
The matrix with entries 
$\Theta_{i,j} =\sum_{k=1}^{m_1}\theta_k^j \theta_k^i$ measures the ellipticity of the system.

In section \ref{averaging2} we state an averaging principle for this system of slow-fast equations.
 
 \subsection{Completely integrable stochastic Hamiltonian systems}
 
In \cite{Li-averaging}  a completely integrable Hamiltonian system (CISHS) in an $2n$ dimensional symplectic manifolds  is introduced, which has $n$  
Poisson commuting Hamiltonian functions. After some preparation this reduces to a slow-fast system in the action angle components. 

We begin comparing this model with the very  well studied random perturbation problem $dx_t=(\nabla H)^\perp(x_t)dt+\epsilon dB_t$ where $B$ is a real valued Brownian motion, $ H:\R^2\to \R$, and $(\nabla H)^\perp$ is the skew gradient of $H$. In the more recent CISHS model,   the energy function is assumed to be random and of the form $ \dot B_t$ so we have the equation $dx_t=(\nabla H)^\perp(x_t)\circ  dB_t$.  In both cases $H(x_t)=H(x_0)$ for all time, so $H$ is a conserved quantity for the stochastic system. Suppose that the CISHS system is perturbed by a small vector field, we have the family of equations
$$ dx_t^\epsilon =(\nabla H)^\perp(x_t^\epsilon)\circ  dB_t+\epsilon V(x_t)dt.$$
Given a perturbation transversal to the energy
surface of the Hamiltonians, one can show  that the energies converge on $[1, \f 1 \epsilon]$ to the solution of a system of ODEs. 
If moreover the perturbation is Hamiltonian, the limit is a constant and one may rescale time and find an effective Markov process on the scale $1/\epsilon^2$.
The averaging theorem was obtained from studying a reduced system of slow and fast variables.
 The CISHS  reduces to a system of  equations in $(H, \theta)$, the action angle coordinates,  where $H\in \R^n$ is the slow variable and $\theta \in S^n$ is the fast variables.
$$\begin{aligned}\f d {dt} H_t^i&=\epsilon f(H_t^\epsilon, \theta_t^\epsilon),\\
d\theta_t^i &=\sum_{i=1}^n X_i(H_t^\epsilon, \theta_t^\epsilon) \circ dW_t^i+\epsilon X_0(H_t^\epsilon, \theta_t^\epsilon)dt.
\end{aligned}$$
This slow-fast system falls, essentially,  into the scope of the article.

\section{ Ergodic theorem for  Fredholm operators depending on a parameter}
\label{Birkhoff}

Birkhoff's theorem for a sample continuous Markov process is directly associated to the solvability of the elliptic differential equation $\L u=v$ where $\L$ is the diffusion operator (i.e. the Markov generator ) of the Markov process and $v$ is a given function. A function $v$ for which $\L u= v$ is solvable should satisfy a number of independent constraints.
The index of the operator $\L$ is  the dimension of the solutions for the homogeneous problem minus the dimension of the independent constraints. 

\begin{defn}
A linear operator $T: E\to F$, where $E$ and $F$ are Hilbert spaces, is said to be a Fredholm operator
if both the dimensions of the kernel of $T$ and the dimension of its cokernel $F/\range(T)$
 are finite dimensional.
The Fredholm index of a Fredholm operator $T$  is defined to be
$$\index (T)=\dim (\ker (T))-\dim ({\mathrm {cokernel}}(T)).$$
\end{defn}
A Fredholm operator $T$ has also closed range and $E_2/\range(T)=\ker(T^*)$. 

A smooth elliptic diffusion operator on a compact space is Fredholm. It also  has a unique invariant probability measure. The Poisson equation $\L u=v$ is solvable for a function $v\in L^2$ if and only if 
$v$ has  null average with respect to the invariant measure, the latter  is the centre condition used in diffusion creations.

If we have a family of operators $\{\L_x:  x\in N\}$ satisfying H\"ormander's condition where $x$ is a parameter taking values in a manifold $N$, the parameter space is typically the state space for the slow variable, we  will need a continuity theorem on the projection operator $f\mapsto \bar f$.
We give a theorem on this in case each $\L_x$ has a unique invariant probability measure. It is clear that for each bounded measurable function~$f$, $\int f(z) d\mu_x(z)$ is a function of $x$. We study its  smooth dependence on~$x$.

   For the remaining of the section, for $i=0, 1, \dots, m$,  let   $Y_i: N\times G\to TG$ be smooth vector fields and  let $\L_x=\f 12\sum_{i=1}^m Y_i^2(x, \cdot)+Y_0(x, \cdot)$.

 \begin{defn}
 If  $\L_x$ satisfies H\"ormander's condition,
 let $r(x,y)$ denote the minimum number  for the vector fields and their iterated Lie brackets up to order $r(x,y)$ to span $T_yG$. Let $r(x)=\inf_{y\in G} r(x,y)$. If $G$ is compact, $r(x)$  is a finite number and will be called the rank of $\L_x$. \end{defn}
  
Let $s\ge 0$, let $dx$ denote the volume measure of a Riemannian manifold $G$ and let $\Delta$ denote the Laplacian. 
 If $f$ is a $C^\infty$ function we define its Sobolev norm to be
 $$\|f\|_{s}= \left( \int_M f(x)(I+\Delta)^{s/2}f(x) \,dx\right)^{\f 12}$$
 and we let $H_s$ denote the closure of $C^\infty$ functions in this norm. This can also be defined without using a Riemannian structure.
 If  $\{\lambda_i\}$ is  a partition of unity  subordinated to a system of coordinates $\{\phi_i, u_i\}$, then the above Sobolev norm
 is equivalent to the norm $\sum_i \| (\lambda_i f)\circ \phi_i\|_{s}$. For a compact manifold, the Sobolev spaces are independent of the choice of the Riemannian metric. 
Let us denote by $|T|$ the operator norm of a linear map $T$. 

Suppose that $\L_x$ satisfies H\"ormander's condition. Let  us re-name the vector fields 
$Y_i$ and  their  iterated Lie brackets  up to order $r(x)$
as $\{Z_k\}$.
Let us define the quadratic form $$Q^x(y)(df,df)=\sum_{i} |df(Z_i(x,y))|^2.$$
Then $Q^x(y)$ measures the sub-ellipticity of the operator. Let 
$$\gamma(x)=\inf_{|\xi|=1} Q^x(y)(\xi, \xi).$$ Then $\gamma(x)$ is locally bounded from below by a positive number.

We summarise the properties of H\"ormander type operators in the proposition below. Let $\L_x^*$ denote the $L_2$ adjoint of $\L_x$. An invariant probability measure  for $\L_x$ is a probability measure such that $\int_G \L_x f (y)\mu_x(dy) =0$ for any $f$ in the domain of the generator. 

 \begin{prop}\label{uniform-sub}
 Suppose that  each $\L_x$  satisfies H\"ormander's condition and that $G$ is compact. 
 Then the following statements hold.
\begin{enumerate}
\item[(1)]  There exists a positive number $\delta(x)$ such that  for every $s\in \R$ there exists a constant $C(x)$  such that for all $u\in C^\infty(G; \R)$
the following sub-elliptic estimates hold,
   $$\|u\|_{s+\delta}  \le C( \|\L_x u\|_s+|u|_{L_2} ), \quad \|u\|_{s+\delta}  \le C( \|\L_x^* u\|_s+|u|_{L_2} ).$$
   We may and will choose $C(x)$ to be  continuous and $\delta(x)$ to be locally bounded from below.
   If $r$ is bounded there exists $\delta_0>0$ such that $\delta(x)\ge \delta_0$.
   \item [(2)]  $\L_x$ and $\L_x^*$ are hypo-elliptic.
   \item [(3)] $\L_x$ and $\L_x^*$ are Fredholm and index=0.
   \item[(4)] If the dimension of $\ker(\L_x)$ is $1$, then $\ker(\L_x)$ consists of constants.
\end{enumerate}
 \end{prop}
 
 \begin{proof}
  It is clear that H\"ormander's condition still holds if we change the sign of the drift $Y_0$, or add a 
 zero order term, or add a first order term which can be written as a linear combination of 
 $\{Y_0, Y_1, \dots, Y_m\}$.  Since
 $$\L_x^*=\f 12 \sum_{i=1}^m (Y_i)^2-Y_0-\sum_i \div(Y_i)Y_i +\div(Y_0)-\f 12 \sum_i L_{Y_i} \div(Y_i)  +\f 12\sum_i [\div (Y_i)]^2, $$
 $\L_x$ satisfies also H\"ormander's condition.

By a theorem of H\"ormander in \cite{Hormander-acta},   there exists a positive number $\delta(x)$, such that for every $s\in \R$ and all $u\in C^\infty(G; \R)$,
   $$\|u\|_{s+\delta(x)}  \le C(x)( \|\L_x u\|_s+|u|_{L_2} ).$$ 
  The constant $C(x)$ may depend on $s$, the $L_\infty$ bounds on the vector fields and their derivatives, and on the rank $r(x)$, and the sub-ellipticity constant $\gamma (x)$.
  The constant $\delta(x)$ in the sub-elliptic estimates depend only on how many number of brackets are needed for obtaining a basis of the tangent spaces, we can for example take $\delta(x)$ to be  $\f 1 {r(x)}$. The number of brackets needed  to obtain a basis at $T_yG$  is upper semi-continuous in $y$ and is bounded for a compact manifold. 
Since  $\L_x$  varies smoothly in $x$, then for $x \in D$ there is a uniform upper bound on the number of brackets needed. Also as indicated in  H\"ormander's proof  \cite{Hormander-acta}, the constant $C(x)$ depends smoothly on the vector fields.
 If there exists a number $k_0$ such that $r(x)\le k_0$ for all $x$, then we can choose a positive $\delta$ that is independent of $x$. This  proves the estimates in part (1) for both $\L_x$ and $\L_x^*$.
 The hypo-ellipticity of $\L_x$ and $\L_x^*$ is the celebrated theorem of H\"ormander and follows from his sub-elliptic estimates, this is part (2). 
  
For part (3) we only need to work with $\L_x$. We sketch a proof for  $\L_x$ to be Fredholm as a bounded operator from  its domain with the graph norm  to $L_2$.  From the sub-elliptic estimates it is easy to see that $\L_x$ has compact resolvents and that $\ker(\L_x)$ and $\ker(\L_x^*)$ are finite dimensional. Then a standard argument shows that $\L_x$ has closed range: If $\L_x f_n$ converges in $L_2$, then either the sequence $\{f_n\}$ is bounded in which case they are also bounded in $H_{\delta}$ the latter is compactly embedded in $L_2$,  and therefore has a convergent sub-sequence. Let us denote  $g$ a limit point. Then since $\L_x$ if closed, $g$ satisfies that $\L g=\lim_{n\to \infty} \L_x f_n$. If $\{f_n\}$ is not $L_2$ bounded, we can find another sequence $\{g_n\}$ in the kernel of $\L$
such that $f_n-g_n$ is bounded to which the previous argument produces a convergent sub-sequence. The dimension of the cokernel is the dimension of the kernel of $\L_x^*$,  proving the Fredholm property. That it has zero index is another consequence of the  sub-elliptic estimates and can be proved from the definition and is an elementary (using  properties of the eigenvalues of the resolvents and their duals), see \cite{Yosida}.
Part (4) is clear as constants are always in the kernel of $\L_x$.
  \hfill$\square$
\end{proof}

If $\mu_1$ and $\mu_2$ are two probability measures on a metric space $M$ we denote by 
$|\mu -\nu|_{TV}=\sup_{A\in \F} |\mu(A)-\nu(A)|$ their total variation norm and $W_1$ their Wasserstein distance:
$$W_1(\mu_1, \mu_2) =\inf_{\nu} \int_{M\times M}  \rho(x,y) \nu(x,y)$$
where $\rho$ is the distance function and the infimum is taken over all couplings of $\mu_1$ and 
$\mu_2$. Suppose that  $\L_x$ has an invariant probability measure $\mu_x(dy)=q(x,y)dy$. If  for a constant $K$, 
$|q(x_1, y)-q(x_2, y)|\le K \rho(x_1, x_2)$ for all $x_1\in M, x_2\in M, y\in G$, 
then 
$|\mu^{x_1}-\mu^{x_2}|_{TV}\le K \rho(x_1,x_2)$.

Let $\mu_x$ be an invariant probability measure for $\L_x$.
We study the regularity of the densities of the invariant probability measures with respect to the parameter, especially the continuity of the  invariant probability measures in the total variation norm. 
This can be more easily obtained if $\L_x$  are  Fredholm operators on the same Hilbert space and if there is a uniform estimate on the resolvent. For a family of uniformly strict elliptic operators, these are possible.

 \begin{rem}
 For the existence of an invariant probability measure, we may use  Krylov-Bogoliubov theorem which is valid  for a  Feller semi-group: Let $P_t(x, \cdot)$ be the transition probabilities. If  for some probability measure $\mu_0$ and for a sequence of  numbers $T_n$ with $T_n\to \infty$, 
 $\{Q_n(\cdot)=\f {1} {T_n} \int_0^{T_n} \int_M P_t (x, \cdot)  d\mu_0(x) dt, n\ge1\}$ is tight,
   then any limit point is an invariant probability measure.  The existence of an invariant probability measure is trivial for a Feller Markov process on a compact space.  Otherwise, a Laypunov function is another useful tool. See \cite{Eckmann-Hairer, Hairer-Mattingly, Hairer-Pillai,Hairer-Mattingly-Scheutzow} for  relevant existence and uniqueness theorems.
 \end{rem}

\begin{rem}
Our  operators $\L_x$ are Fredholm from their domains to $L_2$.
On a compact manifold $\L_x$ is a bounded operator from $W^{2,2}$ to $L_2$ but this is only an extension of  $\L_x$, where $W^{2,2}$ denotes the standard Sobolev space of functions, twice weakly differentiable with derivatives in $L_2$.  We have 
$W^{2,2}\subset \Dom(\L_x)\subset W^{\delta(x), 2}$.
Due to the directions of degeneracies the domain of $\L_x$, given by its graph norm, can be larger than $W^{2,2}$. Since the points of the degeneracies of $\L_x$ move, in general, with $x$,  their domain also change with $x$. 
Suppose that $\L_x$ has zero Fredholm index, then $\L_x$ is an isometry from 
$[\ker(\L_x)]^\perp$ to its image and $\L_x^*$ is invertible on $N^\perp$, the annihilator of the kernel of $\L_x$. Set
$$A(x)=\left|(\L_x^*)^{-1}_{N_x^\perp}\right|_{op}.$$
\end{rem}

In the following proposition  we consider the continuity of $\mu^x$. 

 \begin{prop} \label{continuity-invariant-measure}
Let $G$ be compact. Suppose that $Y_i\in BC^\infty$ and  the conclusions of Proposition \ref{uniform-sub}.  Suppose also that each $\L_x$  has a unique invariant probability measure $\mu^x(dy)$.
 \begin{enumerate}
 \item [(i)] Let $q(x,y)$ denote the kernel of $\mu^x(dy)$. Then $q$  and its derivatives in $y$
are locally  bounded in $x$.\\  If the rank $r$ is bounded from above,  $\gamma$ is bounded from below, then $q$ and its derivatives in $y$ are bounded, i.e. $\sup_x|\nabla^{(k)} \rho(x, \cdot)|_\infty$ is finite for any $k\in \N$.
\item[(ii)] 
The kernel $q$ is smooth in both variables. 
\item [(iii)]
Let $D$ be a compact subset of $N$. There exists a number $c$ such that 
for any $x_1, x_2\in D$, $|\mu^{x_1}-\mu^{x_2}|_{TV}\le c \rho(x_1,x_2)$.
\item [(iv)] Suppose furthermore that $r$ is bounded from above,  $\gamma$ is bounded from below, and $A$ is bounded, then $\mu^x$ is globally Lipschitz continuous in $x$ and $q\in BC^\infty(N\times G)$.
 \end{enumerate}
     \end{prop}
 \begin{proof}
 Each function $q$ solves the equation $\L_x^*q=0$ where $\L_x^*$ is the $L^2$ adjoint of $\L_x$. 
Since $\L_x^*$ is hypo-elliptic, then for each $x$, $q(x, \cdot)$ is $C^\infty$. In other words,  $q(x,\cdot)$  is a function from $M$ to $C^\infty (G, \R)$.  
We observe that $q(x,\cdot)$ are probability densities, so bounded in $L^1$.  If we take $s$ to be a number smaller than $-n/2$, $n$ being the dimension of the manifold, then
$|q(x, \cdot)|_s \le  C|q(x, \cdot)|_{L^1(G)}$.
%Since $G$ is compact, $\sup_x\int_G q^2(x,y)dy$ is bounded. 
We apply the sub-elliptic estimates in part (1) of Proposition \ref{uniform-sub} to $q$:
$$\|u\|_{{s+\delta(x)}} \le c_0(x)(\|\L_x^* u\|_{s}+\|u\|_{s}),$$
 where $\delta(x)$ and $c(x)$ are constants, and obtain that $|q(x, \cdot)|_{s+\delta(x)}\le C(x)$.
 Iterating this we see that for all $s$,
 $$|q(x, \cdot)|_s\le C(\delta(x), r(x), \gamma(x), Y).$$
 The function $C(x)$ depends on  the $L^\infty$ norms of the vector fields $Y_i$ and their covariant derivatives, and also on $\gamma(x)$.
 Also,  $\delta$ can be taken to be  $\f 1 {r(x)+1}$ and $r(x)$ is locally bounded.
By the Sobolev embedding  theorems, $q$  and the norms of  its derivatives in $y$
are locally  bounded in~$x$.
 (If furthermore $r$ and $\gamma$  are bounded, $Y_i$ and their derivatives in $x$  are bounded, then both $\delta$ and $C$ can be taken as a constant, in which case $q$ and their derivatives in $y$ are bounded.)

Since $q$ is in $L^1$, its distributional derivative in the $x$-variable exists and will be denoted by $\partial_x q$.
For each $x$, $\L_x^*q=0$, and so the distributional derivative in $x$ of $\L_x^*q$ vanishes and
$$\partial _x( \L_x^*  )  q(x,y)+\L_x^* \partial_x q(x,y)=0.$$
Set
 $$g(x,y)=-\left(\partial_x(\L_x^*) \right)q(x,y).$$
 Then $g$ is smooth in $y$, whose Sobolev norms in $y$ are locally bounded in $x$.
Since the distributional derivative of $q$ in $x$ satisfies  $\int_G \L_x^*(\partial_x q)(x,y)\,dy=0$
for every~$x$,
 $\int_G g(x,y) \;dy$ vanishes also. Since the index of $\L_x$ is zero, the invariant measure is unique, the dimension of the kernel of $\L_x$ is $1$. The kernel
consists of only constants and so  $g(x,\cdot)$ is an annihilator of  the kernel of $\L$. 
By Fredholm's alternative, this time applied to $\L_x^*$, we see that for  each $x$ we can solve the Poisson equation
$$\L_x^* G(x,y)=g(x,y).$$
Furthermore, by the sub-elliptic estimates, $|G(x,\cdot)|_{L_2(G)}\le A(x) |g(x,\cdot)|_{L_2(G)} $ for some number $A(x)$. Since $A$ is locally bounded, then $G(x,y)$ has distributional derivative in
$x$. But $\partial_x q(x,y)$ also solves $\L_x^*\partial_x q(x,y)=g(x,y)$, by the uniqueness of solutions we see that $\partial_x q(x,y)=G(x,y)$. Thus the distributional derivative of $q$ in $x$ is a locally  integrable function.
 Iterating this procedure and use sub-elliptic estimates to pass to the supremum norm we see that $q(x,y)$ is $C^\infty$ in $x$ with its derivatives in $x$ locally bounded, in particular for a locally bounded function $c_1$, 
 $$\sup_{y\in G} |\partial_x q(x,y)|\le  A(x) c_1(x).$$
  Finally, let $f$ be a measurable function with $|f|\le 1$.  Then
 $$\left| \int_G f(y) q(x_1,y) dy-\int_G f(y) q(x_2,y) dy\right|\le \sup_{x\in D} A(x) \sup_{x\in D} c_1(x)\;\rho(x_1, x_2), $$
 where $D$ is a relatively compact open set containing a geodesic passing through $x_1$ and~$x_2$.
 We use the fact that the total variation norm between two probability measures $\mu$ and $\nu$ is $\f 12 \sup_{|g|\le 1}\left |\int g\,d\mu -\int g\; d\nu\right|$ where the supremum is taken over the family of measurable functions with values in $[-1, 1]$  to conclude  that $|\mu^{x_1}-\mu^{x_2}|_{TV} \le \sup_{x\in D} A(x)  \;\rho(x_1, x_2)$ and conclude the proof. 
  \hfill$\square$
\end{proof}
\begin{exa}
An example of a fast diffusion satisfying all the conditions of the proposition is  the following on $S^1$
and take $x\in \R$:
$$dy_t=\sin(y_t+x) dB_t+\cos(y_t+x) dt.$$ 
Then $\L_x= \cos(x+y) \f{\partial } {\partial y} +\f 12 \sin^2 (x+y) \f{\partial^2 } {\partial y^2}$
satisfies H\"ormander's condition, has a unique invariant probability measure and $r(x)=1$.
Furthermore the resolvent of $\L_x$ is  bounded in $x$. 
\end{exa}

  \begin{defn}
  The operator $\L_x$ is said to satisfy the parabolic H\"ormander's condition if 
$\{Y_1(x,\cdot), \dots, Y_{m_2}(x, \cdot)\}$ together with the brackets and iterated the brackets of 
$\{Y_0(x, \cdot), Y_1(x,\cdot), \dots, Y_{m_2}(x, \cdot)\}$ spans the tangent space of $N$ at every point.  Let $P^x(t, y_0, y) $ denote the semigroup generated by $\L$. 
  \end{defn}

\begin{remark}
Suppose that each $\L_x$ is  symmetric, satisfies the parabolic H\"ormander's condition and the following
uniform Doeblin's condition:  
there exists a constant $c\in (0,1]$,  $t_0>0$, and a probability measure $\nu$ such that 
$$P_t^x(y_0, U)\ge c\nu(U),$$
for all $x\in N$, $y_0\in G$ and for every  Borel set $U$ of $G$,
Suppose that $Y_j\in BC^\infty$.  Then $A(x)$ is bounded. 
In fact for any $f$ with $\int f(y) \mu^x(dy)=0$, the function $P_t^xf(y_0)=\int_G f(y) P^x(t,y_0, dy)$ converges to $0$ as $t \to \infty$ with a uniform exponential rate. Since $\L_x$ satisfies the parabolic H\"ormander's condition, $\L-\f {\partial}{\partial x}$ satisfies H\"ormander's condition on $M\times \R$. Then by the sub-elliptic estimates for $\L-\f {\partial} {\partial t}$,
$P_t^xf$ converges also in $L_2$. Let $R^x$ denote the resolvent of $\L_x$. Since $$\<R^x f, f\>_{L_2}=\int_G \int_0^\infty P_t^xf (y)f(y)\, dt\, dy,$$
then $R^x$ is uniformly bounded. Since $\L_x$ is symmetric, this gives a bound on $A(x)$.
We refer to the book \cite{Bakry-Gentil-Ledoux} for studies on Poincar\'e inequalities for Markov semi-groups. 
\end{remark}

\begin{cor} Let $G$ be compact.
Then $q$ is  smooth in both variables and in $BC^\infty(N\times G)$. In particular $\mu^x=q(x,y)dy$ is globally Lipschitz continuous.
\end{cor}
 Just note that the semigroups $P_t^x$ converges  to equilibrium with uniform rate.  The spectral gap of $\L_x$  is bounded from below by a positive number.  

The following is a version of the law of large numbers.

\begin{thm}\label{lemma-lln}
Let $G$ be compact. Suppose that $\sum_{j=1}^{m_2}|Y_j|_\infty$ is finite, and the conclusions of Proposition \ref{uniform-sub}. Suppose that  each $\L_x$ 
  has a unique invariant probability measure $\mu_x$.    
 
 Let $s>1+\f{\dim(G)}{2}$. Then there exists a constant $C(x)$ such that for every $x\in N$
and for every smooth real valued function $f:N\times G \to \R$ with compact support in the first variable (or independent of the first variable), 
\begin{equation}
\label{llln-1}
\sqrt{\EE \left({1\over T} \int_t^{t+T} f(x, z_r^x) \;dr- \int_G f(x, y) \mu_x(dy)\right)^2} \le C(x)\|f(x, \cdot)\|_s{1\over \sqrt{T}}
\end{equation}
where  $z_r^x$ is  an $\L_x$ diffusion and $C(x)$ is locally bounded. 
\end{thm}

\begin{proof}
 In the proof we take $t=0$ for simplicity.
We only need to work with a fixed $x\in N$. We may assume that  $\int_G f(x, y) \mu_x(dy)=0$.
 Since $\L_x$ is hypo-elliptic  and since $\mu_x$ is the  unique invariant probability measure then,  for any smooth function $f$ with $\int_G f (x,y)\, \mu_x(dy)=0$, $\L_x g (x,\cdot)=f(x, \cdot)$ has a smooth solution. If $f$  is compactly supported in the first variable, so is $g$.
We may then apply It\^o's formula to the smooth function $g(x, \cdot)$, allowing us to estimate
$\f 1 T \int_0^T f(y_r^x) dr$ whose $L^2(\Omega)$ norm is controlled by the norm of $g$ in $C^1$
and the norms  $|Y_j(x, \cdot)|_\infty$.
The $\L_x$ diffusion satisfies the equation:
 $${1\over T} \int_0^T f(x, z_r^x)dr
={1\over T} \left(g(x, z_T^x)-g(x, y_0)\right)
-{1\over T}\left(\sum_{k=1}^{m_2}  \int_0^T dg(x, \cdot)(Y_k(x, z_r^x) ) dW_r^k\right).$$
Since $|Y_j(x, \cdot)|_\infty$ is bounded,  it is sufficient to estimate the stochastic integral term by Burkholder-Davis-Gundy inequality:
$$\EE\left(\sum_{k=1}^{m_2}  \int_0^T dg(x, \cdot)(Y_k(x, z_r^x) ) dW_r^k\right)^2
\le   \sum_{k=1}^{m_1}  |Y_k|_\infty^2\int_0^T\EE |dg(x, z_r^x)|^2 \,ds.$$
 It remains to control the  supremum norm of
$dg(x, \cdot)$. By the Sobolev embedding theorem this is controlled by  the $L_2$ Sobolev norms  $\|f(x, \cdot)\|_s$ where  $s>1+\f {\dim (G)} 2$. Let $D$ be a compact set  containing the supports of the functions $f(\cdot,y)$.  We can choose a number $\delta>0$, chosen according to $\sup_{x\in D} r(x)$, such that the sub-elliptic estimates holds for every $x\in D$.
 There exist constants  $c_1, c_2, c_3$ such that for every $x\in D$, \begin{equation*}
{\begin{split}
|dg(x, \cdot)|_{\infty}\le c_1 \,\|g(x, \cdot)\|_{s+\delta}\le c_2(x) \; ( \|f(x,\cdot)\|_{s}+|g(x, \cdot)|_{L_2})
\le c_3(x)\, \|f(x,\cdot)\|_{s}.
\end{split}}
\end{equation*}\
The constant $c_2$ may depend on $s$. The constant $c_3(x)$ is locally bounded. We have used the following fact. The spectrum of $\L_x$ is discrete, the dimension of the kernel space of $\L_x$ is $1$ and hence  the only solutions to $\L_x  h=0$ are constants. We know that the spectral distance is continuous, which is not the right reason for $c_3(x)$ to be locally bounded. To see that we may assume that $f$ is not a constant and observe that $|\L_x^{-1}|_{op}\le k(x)$ where $k(x)$ is a finite number.
This number is locally bounded following the fact that the semi-group $P_t^x f$ converges to zero exponentially and the kernels for the probability distributions of $\L_x$ are smooth in the parameter $x$.
  \hfill$\square$
\end{proof}

For the study of the limiting process in stochastic averaging we would need to know the regularity of 
the average of a Lipschitz continuous function with respect to one of its variables. The following illustrates what we might need.

\begin{prop}\label{proposition-Lipschitz}
Let $\{\mu^x, x\in M\}$ be a family of  probability measures on $G$. Let $f: N\times G\to \R$ be a measurable function.
\begin{enumerate}
\item 
[(1)] Let $f$ be a bounded function, Lipschitz continuous in the first variable, i.e.
$|f(x_1, y)-f(x_2, y) | \le K_1(y) \rho(x_1, x_2)$ with $\sup_{x\in M} |K_1|_{L_1(\mu_x)}<\infty$.
Then $$\left|\int_G f(x_1,y )\, \mu^{x_1}(dy) -\int_G f(x_2, y)\, \mu^{x_2}(dy)\right|\le K_2\,\rho(x_1, x_2)
 +|f|_\infty |\mu^{x_1}-\mu^{x_2}|_{TV}.$$
 \item [(2)] Suppose furthermore that $\mu_x$ depends continuously on $x$ in the total variation metric.  Let $f$ be bounded continuous such that
 $$|f(x_1,z)-f(x_2,z)|\le K_3\,\rho(x_1, x_2), \quad \forall z\in G, x_1,x_2\in M,$$
for a positive number $K_2$.  Then $\int_0^T \int_G f( x_s, z) \mu^{ x_s} (dz) ds$ exists, and if $D$ is the support of $f$ then
$$\begin{aligned}
&\left|\sum_{i=0}^{N-1}\Delta t_i\int_G f( x_{t_i}, z) \; \mu^{ x_{t_i}}(dz)-\int_0^T \int_G f( x_s, z) \mu^{ x_s} (dz) ds\right|\\
\le& TK_3\sup_{0\le i<N-1} \sup_{s\in [t_i, t_{i+1})} [ \rho(x_s, x_{t_i})]
+|f|_\infty\,\cdot \sup_{0\le i<N-1}  \sup_{s\in [t_i, t_{i+1})} \left(\left| \mu^{ x_s}-  \mu^{x_{t_i}}\right|_{TV}\chi_{x_s\in  D}\right).
\end{aligned}$$
\item[(3)]
Suppose that $\mu_x$ depends continuously on $x$ in the Wasserstein 1-distance. Then for 
any  bi-Lipschitz continuous $f$,
 $\int_0^T \int_G f( x_s, z) \mu^{ x_s} (dz) ds$ exists and the estimate 
 in part (1) holds with the total variation distance replaced by $W_1$, the Wasserstein 1-distance.
\end{enumerate} 
\end{prop}

\begin{proof}
Just observe that:
$$\begin{aligned}
&\left|\int_G f(x_1,y )\, \mu^{x_1}(dy) -\int_G f(x_2, y)\, \mu^{x_2}(dy)\right|\\
 &\le \int K_1(y) \mu^{x_1} (dy)\,\rho(x_1, x_2)
 +|f|_\infty |\mu^{x_1}-\mu^{x_2}|_{TV},
\end{aligned}$$
obtaining the required inequality in part (1) .
For any non-negative numbers $s,t$, 
$$\begin{aligned}
&\left| \int_G f( x_{t}, z) \; \mu^{ x_{t}}(dz)- \int_G f( x_s,z) \mu^{ x_s} (dz) \right|\\
&  \le K_3 \; \rho(  x_{t},  x_s) + \left|\int_G f(x_s, z) \mu^{ x_{t}}(dz)-\int_G f(x_s, z)\mu^{ x_s}(dz)\right|.\end{aligned}$$
This holds pathwise. Since each function  $f(x_s(\omega, \cdot))$ is  bounded by $|f|_\infty$,  $$\left| \int_G f( x_{t}, z) \; \mu^{ x_{t}}(dz)- \int_G f( x_s,z) \mu^{ x_s} (dz) \right|
 \le K_3 \; \rho(  x_{t},  x_s)
+|f|_\infty \;  \left| \mu^{ x_s}-  \mu^{x_{t}}\right|_{TV} \chi_{x_s\in D}.$$
Since $x_s$ is sample continuous, $x\mapsto \mu^x$ is continuous and $f$ is a bounded and continuous, $\int_G f(x_s, z)\mu^{ x_s}(dz)$ is continuous in $s$ and so integrable in $s$.
Consequently,
$$\begin{aligned}
&\left|\sum_{i=0}^{N-1}\Delta t_i\int_G f( x_{t_i}, z) \; \mu^{ x_{t_i}}(dz)-\int_0^T \int_G f( x_s, z) \mu^{ x_s} (dz) ds\right|\\
&\le  \sum_{i=0}^{N-1} \Delta t_i K_3  [ \rho(x_s, x_{t_i})]
 + \sum_{i=0}^{N-1}  \Delta t_i |f|_\infty  [\chi_{x_s\in D}| \mu^{ x_s}-  \mu^{x_{t_i}}|_{TV}] \\
&\le TK_3 \sup_{s\in [t_i, t_{i+1})}\EE [ \rho(x_s, x_{t_i})]
+|f|_\infty
\sup_{s\in [t_i, t_{i+1})}  \left(\chi_{x_s\in D} \left| \mu^{ x_s}-  \mu^{x_{t_i}}\right|_{TV}\right).
\end{aligned}$$
Finally   we use the fact that
$f$ is Lipschitz in the second variable and  the following dual formulation for the Wasserstein 1-distance $W_1(\mu, 
\nu)$ of two probability measures $\mu$ and~$\nu$,
$$W_1(\mu, \nu)=\sup_{|g|_{\Lip=1} }\left| \int g \,d\mu -\int g\, d\nu\right|,$$
where $|g|_{\Lip}$ denotes the Lipschitz constant of $g$. We obtain
$$\begin{aligned}
\left| \int_G f( x_{t}, z) \; \mu^{ x_{t}}(dz)- \int_G f( x_s,z) \mu^{ x_s} (dz) \right|
\le  K_3 \; \rho(  x_{t},  x_s)
+K_4 \;  W_1\left( \mu^{ x_s}, \mu^{x_{t}}\right). \end{aligned}$$
  The required assertion and  estimate now follows by  the argument in part~{(2)}.
\hfill$\square$
\end{proof}% 
Put Proposition \ref{continuity-invariant-measure} and Proposition \ref{lemma-lln} together we
obtain Theorem \ref{lln}. 

Finally we would like to refer  to \cite{Bally-Caramellino} for the convergence  in total variation in the Law of large numbers for independent random variable, see also \cite{Korepanov-Koslof-Melbourne}. See the books \cite{Elworthy-LeJan-Li-book, Baudoin} for stochastic flows in sub-Riemmian geometry. It would be interesting to study problems in this section under the `uniformly finitely generated' conditions, see e.g.    \cite{Crisan-Ottobre, Kusuoka-Stroock-II}. See also
\cite{Albeverio-Daletskii-Kalyuzhnyi, Cass-Friz}.

 \section{Basic Estimates for  SDEs on manifolds}
 \label{estimates}
  To obtain an averaging theorem associated to a family of stochastic processes $\{x_t^\epsilon, \epsilon>0\}$ on a manifold $N$, we  first prove that the family of stochastic processes is pre-compact and we then proceed to identify the limiting processes. To this end we first obtain uniform estimates on the family of slow variables,  on the space of continuous functions on the manifold,   and also obtain estimates on the  limiting Markov processes.
  In  this section we  obtain essential estimates for a  general  SDE and these estimates will be in terms of bounds on the driving vector fields.

Throughout this section we assume that $M$ is a connected smooth and complete Riemannian manifold, $B_t=(B_t^1, \dots, B_t^m)$ is an $\R^m$-valued  Brownian motion. Let $X_0$ be a  vector field and $X: M\times \R^m\to TM$ be a map linear in the second variable. 
For $x\in M$, let  $\phi_t(x)$ denote the solution to the SDE
\begin{equation}\label{sde}
d x_t= \sum_{k=1}^m X_k(x_t)\circ dB_t^k+X_0(x_t) \,dt,
\end{equation}
with initial value $x$. We also  set $x_t=\phi_t(x_0)$.

The type of estimates we need are variation of the following $\EE[ \rho(x_s, x_t)]^2 \le C|t-s|$ where the constant 
$C$ depends on the SDE only on specific bounds for the driving vector fields. Since no ellipticity is assumed,  it is essential to deal with the problem that $\rho(x,y)$ is only $C^1$, when $x$ and $y$ are on the cut locus of each other, and we cannot apply It\^o's formula to $\rho$ directly.   If we are only interested in obtaining tightness results, this problem can be overcome by choosing an auxiliary distance function.  Otherwise, e.g. for the convergence of the stochastic processes,  we  work with the Riemannian distance function $\rho: M\times M\to \R$. Let $M\times M$ be given the product Riemannian metric.  Let  $|f|_\infty$  denote the $L_\infty$ norm of a function~$f$.

\begin{lem}\label{lemma2-1}
\begin{enumerate}
\item[(1)]  Suppose that $M$ is a complete Riemannian manifold  with bounded sectional curvature. Then for each $\delta>0$ there exists a smooth distance like function $f_\delta: M\times M\to \R$  and a constant $K_1$ independent of $\delta$  such that
$$\left| f_\delta-\rho\right|_\infty
\le \delta, \quad |\nabla f_\delta| \le K_1, \quad  |\nabla^2 f_\delta |\le K_1.$$
If furthermore the curvature has a bounded covariant derivative, then we may also assume that $ |\nabla^3 f_\delta |\le K_1$.
\item [(2)] 
If $M$ is compact Riemannian manifold, there exists a smooth function $f:M\times M\to \R$ such that 
$f$ agrees with $\rho$ on a tubular neighbourhood of the diagonal set of the product manifold 
$M\times M$.
 \end{enumerate}
\end{lem}
\begin{proof}
(1)
 For the distance function $\rho(\cdot, O)$, where $O$ is a fixed point in $M$, this is standard, see \cite{Schoen-Yau,Tam-exhaustion, Chow-3}. To obtain the stated theorem it is sufficient to repeat the proof there for  the 
distance function on the product manifold.  The basic idea is as following. By a theorem of Greene and Wu \cite{Greene-Wu}, every  Lipschitz continuous function  with gradient less or equal to $K$ can be  approximated by $C^\infty$ functions whose gradients are bounded by $K$. 
We apply this to the distance function $\rho$ and obtain for each $\delta$ a smooth function $f_\delta: M\times M\to \R$  such that 
$$|\rho-f_\delta|_\infty \le \delta, \qquad |\nabla f_\delta |_\infty\le 2.$$
We then convolve $f_\delta$ with the heat flow to obtain  $f_\delta(x,y, t)$,  apply Li-Yau heat kernel estimate for manifolds whose Ricci curvature is bounded from below and using harmonic coordinates on a a small geodesic ball of radius $a/K$  where $K$ is the upper bound of the sectional curvature and $a$ is a universal constant.
For part (ii),   $M$ is compact. We take a smooth cut off function  $h: \R_+\to \R_+$ such that 
$h(t)=1$ for $t<a$ and  vanishes for $t>2a$ where $2a$ is  the injectivity radius of $M$
and such that $|\nabla h|$ is bounded. The function $f:=h\circ \rho$ is as required.
\hfill$\square$
\end{proof}

Set $\tilde X_0=\f 12 \sum_{i=1}^m\nabla X_i(X_i)+X_0$.  We denote by $\rho$ the Riemannian distance on $M$.
Let $T$ be a positive number and let  $O\in M$.  Let $K'$, $K$, $a_i$ and  $b_i$  denote constants.
\begin{lem}\label{lemma2-2}
Suppose that $\tilde X_0$  and $X_i$ are $C^1$, where $i=1,\dots, m$. Suppose {\bf one} of the following  two conditions hold.
\begin{enumerate}
\item [(i)]
The sectional curvature of $M$ is bounded by $K'$, and for every $x\in M$,
 $$|X_i(x)|^2\le K+  K\rho(x, O), \quad |\tilde X_0(x)|\le  K+ K\rho(x, O).$$
 \item [(ii)] Suppose that  $\rho^2:M\times M\to \R$ is smooth and 
 $$\f 12 \sum_{i=1}^m \nabla d \rho^{2p}(X_i, X_i)+ d\rho^{2p} (\tilde X_0)\le   K+K \rho^{2p}.$$
\end{enumerate}
Then, the following statements hold. 
\begin{enumerate}
\item [(a)]There exists a constant $c$ which depends only on  $K'$, $T$, $p$, and $\dim(M)$ such that for every  pair of  numbers $s,t$ with $0\le s\le t\le T$, 
$$\begin{aligned} &\EE    \rho^{2p}(x_t, O)
\le c(Kt+ 1+\rho^{2p}(x_0, O))  e^{cKt}, \\
& \EE  \left\{  \rho^{2p}(x_s,x_t)\Big |\F_s \right\}\le c|t-s|(1+K)e^{cK|t-s|}.\end{aligned} $$
\item[(b)] Suppose that in addition $|X_i|$ is bounded for every $i=1, \dots, m$. Then, for every $p\ge 1$, there exists a constant $C$, which depends only on $p$, $K'$, $m$,  and $\dim(M)$ and a constant $c(T)$,  such that for every $s<t\le T$, 
$$\EE  \left( \sup_{s\le u\le t}  \rho^{2p}(x_s,x_u)\right)\le c+KC(T)\,e^{C(T)K}.$$
Also, $\EE  \left( \sup_{s\le u\le t}  \rho^{2p}(O,x_u)\right)
\le c( \rho^{2p}(O, x_0)+ Kc(T)) \,e^{Kc(T)}$.
\end{enumerate}
\end{lem}

\begin{proof}
Let $\delta \in (0,1]$ and let $f_\delta:M \times M\to \R$ be a smooth function satisfying the estimates
$$\left| f_\delta-\rho\right|_\infty \le \delta, \quad |\nabla f_\delta| \le K_1, \quad  |\nabla^2 f_\delta |\le K_1$$
where $K_1$ is a constant depending on $K'$ and $\dim(M)$. If $\rho^2 $ is smooth we take $f_\delta=\rho$.

(a) Either hypothesis (i) or (ii)  implies that the SDE (\ref{sde}) is conservative.
For any $x\in M$ fixed we apply It\^o's formula  to the second variable of the function $f_\delta^2(x, y)$ on the time interval $[s,t]$:
\begin{equation}\label{2-2-1}
f_\delta^{2p}(x,x_t)=f_\delta^{2p}(x, x_s)+\int_s^t \L f_\delta^{2p}(x, x_r)\, dr+\int_s^t 2 f^{2p-1}_\delta(x, x_r)   (d f_\delta)(X_i(x_r))\,dB_r^i,
\end{equation}
where $d $ and $\L$ are applied to the second variable. 
Let $\tau_n$ denote the first time after~$s$ that  $f_\delta(x,x_t)\ge n$ and 
we  take the expectation of the earlier identity  to obtain
$$\EE[f_\delta^{2p}(x,x_{t\wedge \tau_n} )]=\EE[ f_\delta^{2p}(x, x_s)]+ \int_s^t \EE\left [\chi_{r<\tau_n} \L f_\delta^{2p}(x, x_r)\right]\,dr.$$
Under hypothesis (ii),   we use~$\rho$ in place of $f_\delta$
and conclude  by Gronwall's inequality that $\EE \rho^{2p}(x_t, O)\le (\rho^{2p}(x, O)+Kt)e^{Kt}$. The second estimate follows from Markov property and taking $O=x_s$.

Let us now assume hypothesis (i) and let $C_1, C_2, \dots$ denote a constant depending on $p$.
In the formula below,  $\nabla$ denotes differentiation w.r.t. the second variable,
$$\begin{aligned}
\L[ f_\delta^{2p}] (x,y) =& p(2p-1)\sum_{i=1}^m   f_\delta^{2p-2}(x,y)|\nabla f_\delta  (X_i(y))|^2\\
&+p\sum_{i=1}^m
f_\delta^{2p-1}(x,y) | \nabla ^2 f_\delta (X_i(y), X_i(y))| 
+2f_\delta^{2p-1} (x,y) d f_\delta(\tilde X_0(y)).
\end{aligned}$$
 We first take $x=O$ and $s=0$, to see that
$\L f_\delta ^{2p}(O,y) \le C_1K f_\delta^{2p} (O,y) + C_1K$. We may then apply Grownall's inequality followed by Fatou's lemma to obtain:
$$\EE f_\delta^{2p}(x_t, O)\le (f_\delta^{2p}(x_0, O)+C_1Kt)e^{C_1Kt}.$$
Take $\delta=1$,  we conclude the first estimate from the following inequality:
$$\EE [\rho^{2p}(x_t, O)] 
\le C_2 +C_2 \EE f_1^{2p}(x_t, O)\le C_2 +C_3 (\rho^{2p}(x_0, O)+1+Kt)e^{C_1Kt}.
$$
Let  $s<t$.   Using  the flow property, we see that 
$$\EE  \{ \rho ^{2p}(x_s, x_t)|\F_s\}
\le C_4\delta^{2p}+ C_4\EE  \{ f_\delta ^{2p}(x_s, x_t)|\F_s\} \le 
C_4\delta^{2p}+  C_5K(t-s)e^{C_5K(t-s)}.$$ 
  For any $s,t>0$ we may choose $\delta_0$ such that $\delta _0^{2p}<|t-s|$ and conclude that
$$\EE  \{ \rho ^{2p}(x_s, x_t)|\F_s\}
 \le   C_6(1+K) |t-s|e^{C_6K|t-s|}.$$ 
For part (b)  we take $\delta=1$ and take $p=2$ in  (\ref{2-2-1}). 
 Then
$$\begin{aligned}
\EE \sup_{u\le t} f_1^{2p}(O,x_u)=&
C_1f_\delta^{2p}(O, x_0)+C_1\left( \int_0^t (K+Kf_\delta^{2}(O, x_r)) \, dr\right)^p\\
&+
 C_1\sum_i  \EE\left(\int_0^t 2 f^{2p-1}_1(O, x_r)   (d f_1)(X_i(x_r))dr\right)^p
 \end{aligned} $$
Since  $|X_i|$ is bounded for $i= 1, \dots, m$, 
 $|2 f_1(x, y)   (d f_1)(X_i(y))| \le 2|f_1(x, y) |\cdot |X_i(y)|$.
 We  conclude that 
 $$\EE \sup_{0\le u\le t} f_1^{2p}(O,x_u)\le C_2( f_\delta^{2p}(O, x_0)+KC(T)) e^{KC(T)}.$$
 This leads to the required estimates for $ \EE\left[ \sup_{0\le u \le t} \rho^{2p}(O,x_u)\right]$. Similarly, for  some constants $c_1$ and  $c$, depending on $ m$ and the bound of the sectional curvature, for some constants $c$ and $C(T)$, 
$$\EE\left[ \sup_{s\le u \le t} \rho^{2p}(x_s,x_u)\right]\le c_1+
c_1K\EE\left[ \sup_{s\le u \le t} (f_1)^{2p}(x_s,x_u)\right]\le c+cKC(T) e^{Kc(T)}.$$
We have completed the proof for part (b).
\hfill$\square$
\end{proof}

These estimates will be applied in the next section to both of our slow and fast variables.
For the slow variables, we have the uniform bounds on the driving vector fields and hence we obtain a uniform moment estimate (in $\epsilon$) of the distance traveled by the solutions. For the fast variables, the vector fields are bounded by $\f 1 \epsilon$ and we expect  that the evolution of the $y$-variable in an interval of size $\Delta t_i$ to be controlled by the following quantity
$\f {\Delta t_i} \epsilon e^{\f {\Delta t_i}\epsilon}$.

%\begin{remark}
%We expect that $\nabla^2 \rho^2 \le c+c\rho$ or under the more crude estimate $\nabla^2 \rho^2 \le c+c\rho^2$.
%Observe that  $\L \rho^{2p} =(d\rho^{2p})(\tilde X_0)+\f 12 \nabla d\rho^{2p}(X_i,X_i)$ and so the assumptions on $\L \rho^{2p}$ makes good sense either under the assumption that 
%
%\end{remark}
%\begin{remark}
%By the same reasoning we should expect that
%$\sup_{\delta\in (0,1]}| \nabla ^2 f^2_\delta |\le C$. In this holds, we may replace the conditions on $X_i$ where $i\ge 1$  in  (ii) by $\sum_{i=1}^m|\nabla X_i|^2\le C+C\rho^2(x,O)$. This brings us in line with the Euclidean case.
%\end{remark}

\section{Proof of Theorem 2}\label{proof}
We proceed to prove the main averaging theorem,  this is Theorem 2   in section \ref{results}.

 In this section $N$ and $G$ are smooth complete Riemannian manifolds and  $N\times G$ is the product manifold with the product Riemannian metric.  We use $\rho$ to denote the Riemannian distance on $N$, or on $G$, or on $N\times G$. This will be clear in the context and without ambiguity. For each $y\in G$ let $X_i(\cdot, y) $ be smooth vector fields on $N$ and for each $x\in N$ let 
$Y_i(x, \cdot )$ be smooth vector fields on $G$, as given in the introduction. 
Let $x_0\in N$ and $y_0\in G$. We denote by $(x_t^\epsilon, y_t^\epsilon)$   the solution to the equations:

\begin{equation}\label{sdes}
\left\{\begin{aligned} dx_t^\epsilon&=\sum_{k=1}^{m_1}  X_k(x_t^\epsilon, y_t^\epsilon)\circ dB_t^k+
X_0(x_t^\epsilon, y_t^\epsilon) \,dt, \quad x_0^\epsilon=x_0;\\
dy_t^\epsilon&= \f 1 {\sqrt \epsilon}\sum_{k=1}^{m_2}  Y_k(x_t^\epsilon, y_t^\epsilon) \circ dW_t^k+
\f 1{\epsilon} Y_0(x_t^\epsilon, y_t^\epsilon)\,dt, \quad y_0^\epsilon =y_0.\\
\end{aligned}\right.
\end{equation}
Let us first study the slow variables $\{x_t^\epsilon, \epsilon \in (0,1]\}$. We use $O$ to denote a reference point in $N$.

\begin{lem}\label{lemma4.1}
Under Assumption \ref{assumeX},
 the family of stochastic processes  $\{x_t^\epsilon, \epsilon\in (0, 1]\}$ is tight on any interval $[0,T]$ where $T$ is a positive number.  Furthermore  there exists a number $C$ such that for any $p>0$,
 $$\sup_{\epsilon\in (0,1]}\sup_{s,t\in[0, T]}\EE \rho^2(x_s^\epsilon, x_t^\epsilon)\le C|t-s|,
 \quad \sup_{\epsilon\in (0,1]}\sup_{s, t\in [0, T]} \EE \rho^{2p}(x_s^\epsilon, x_t^\epsilon)<\infty.$$
Any limiting process of $x_t^\epsilon$, which we denote by $\bar x_t$,  has infinite life time and 
 satisfies the same estimates:  $\EE \rho^2(\bar x_s, \bar x_t) \le C(t-s)$ and
  $\sup_{ s, t\in[0, T]} \EE \rho^{2p}(\bar x_s, \bar x_t)$ is finite.
\end{lem}
\begin{proof}
Assumption \ref{assumeX} states that:
the sectional curvature of $N$ is bounded,
 $|X_i(x, y)|^2\le K+ K\rho(x, O)$ and $|\tilde X_0(x,y)|\le  K+K \rho(x, O)$.
Or   $\rho^2:N\times N\to \R$ is smooth,  and $$\f 12 \sum_{i=1}^m \nabla d \rho^2(X_i(\cdot,y), X_i(\cdot, y))+ d\rho^2 (\tilde X_0(\cdot,y ))\le  K+ K\rho^2(\cdot, O).$$
In either case, the bounds are independent of the $y$-variable.
We apply  Lemma~\ref{lemma2-2} to each $x_t^\epsilon$ to obtain estimates that are uniform in $\epsilon$: there exists a constant $C$ such that
for all $0\le s\le t\le T$ and for every $\epsilon>0$,  $\EE \rho^2(x_s^\epsilon, x_t^\epsilon)\le C|t-s|$.
Then use a chaining argument we  obtain the following estimate for  some positive constant $\alpha$:
$\EE \left[\sup_{|s-t|\not =0} \f {\rho(x_t^\epsilon, x_s^\epsilon)} {|t-s|^\alpha}\right]<\infty$, this proves the tightness.  Since $\EE [\rho( x_t^\epsilon, O)^2]$ is uniformly bounded, we see  $x_t$ has infinite lifetime and $\EE [\rho( x_t, O)^2]$ is finite.
 From  the uniform estimates  $\EE \rho^2(x_s^\epsilon, x_t^\epsilon ) \le C(t-s)$
and  $\EE \rho^4(x_s^\epsilon, x_t^\epsilon )^2 \le C(t-s)^2$, we easily obtain $\EE \rho^2(\bar x_s, \bar x_t) \le C(t-s)$
and the other required estimates for $\bar x_s$.
\hfill$\square$.
\end{proof}

Let us fix $x\in N$. For $t\ge s$, let 
$\phi_{s,t}^x(y)$ denote the solution to the equation
\begin{equation}\label{z-slow}
dz_t= \sum_{k=1}^{m_2}  Y_k(x, z_t) \circ dW_t^k+
 Y_0(x, z_t)\,dt, \quad z_s=y.
\end{equation}
Write $z_t^x=\phi^x_{0, t}(z_0)$, its Markov generator  is  $\L_0^x=\f 12 \sum_{k=1}^{m_1} (Y_i(x, \cdot))^2+Y_0(x, \cdot)$.
Let $\phi_{s,t}^{\epsilon,x}$ denote the solution flow to   the SDE:
\begin{equation}
dy_t= \f 1 {\sqrt \epsilon}\sum_{k=1}^{m_2}  Y_k(x, y_t) \circ dW_t^k+
\f 1{\epsilon} Y_0(x, y_t)\,dt, \quad y_s=y_0
\end{equation}
Observe that the time changed solution flow $\phi^x_{\f s \epsilon, \f t\epsilon}(\cdot)$ agrees with
$\phi_{s, t}^{\epsilon,x}(\cdot)$.
On each sub-interval $[t_i, t_{i+1})$  we set 
\begin{equation}\label{time-change}
z_t^{x_{t_i}^\epsilon}=\phi_{\f{t_i}\epsilon, t}^{x_{t_i}^\epsilon}(y_{t_i}^\epsilon),\quad 
 y_t^{x_{t_i }^\epsilon}=\phi_{t_i/\epsilon, \,t/\epsilon }^{x_{t_i }^\epsilon}(y_{t_i}^\epsilon).
\end{equation}
 
  In the following locally uniform law of large numbers (LLN),  any rate of convergence $\lambda(t)$ is allowed.
\begin{assumption}[Locally Uniform LLN]
\label{assumptions-Birkhoff}
Suppose that there exists a family of probability measures $\mu_x$  on $G$ which is continuous in the total variation norm. Suppose that for any smooth function $g:G\to \R$ and  for any initial point $z_0\in G$ and $t_0\ge 0$,
 $$\left|\f 1 t  \EE\int_{t_0}^{t+t_0}   g\left (\phi_{t_0,s}^x(z_0)\right) ds-\int_G g(z) \mu_x(dz) \right|_{L_2(\Omega)}
\le  \alpha(x)\,\|g\|_s\, \lambda(t).
$$
 Here $\lambda(t)$ is a constant such that $\lim_{t\to \infty} \lambda(t)=0$,  $s$ is a non-negative number, and $\alpha(x) $ is a real number  locally bounded in $x$. \end{assumption}
 \begin{remark}
 In Proposition \ref{lemma-lln} we proved that if each $\L_x$ satisfies  H\"ormander's condition and if $\mu_x$ is the invariant probability measure for $\L_x$ (assume uniqueness), the locally uniform LLN holds with  $\lambda(t)=\f 1 {\sqrt t}$.
 \end{remark}
Suppose that  $f:N\times G\to \R$ is bounded measurable, we define $\bar f(x)=\int_G f(x, z) \; \mu^{x}(dz)$.

\begin{lem}
\label{lemma2-3-2}
Suppose the locally uniform LLN assumption. Let $f: N\times G\to \R$ be a smooth function with compact support (it is allowed to be independent of the first variable).  Let $t_0=0<t_1<\dots <t_N=T$ be a partition of equal size $\Delta t_i$. Then,
for some number $c$,
 $$\begin{aligned} &\EE\sum_{i=0}^{N-1}\left|  \int_{t_i}^{t_{i+1}}f\left(x_{t_i}^\epsilon, y^{x_{t_i }^\epsilon}_{r}\right) \,ds- \Delta t_i\, f\left( x_{t_i}^\epsilon\right)  \right|
\le c\, T  \,\lambda(\f {\Delta t_i} \epsilon)\sup_{x\in D}  \,\left \|f(x, \cdot)-\bar f(x)\right\|_{s}.
\end{aligned}$$
\end{lem}
\begin{proof} Set $\bar \alpha=\sup_{x\in D} \alpha(x)$ ant   $C=\sup_{x\in D}  \left \|f(x, \cdot)-\bar f(x)\right\|_{s} $, both are finite numbers by the assumptions on $f$ and on $\alpha(x)$. 
Firstly we observe that
$$\begin{aligned}
&\left| \EE\left\{ \f \epsilon{  \Delta t_i }  \int_{\f {t_i}\epsilon }^{\f {t_{i+1}}\epsilon}
 f(x_{t_i}^\epsilon, y_r^{x_{t_i }^\epsilon})\,dr -\bar  f(x_{t_i}^\epsilon)  \Big|   \F_{t_i} \right\} \right|
 \le  \alpha(x_{t_i}^\epsilon)\, \lambda (\f {\Delta t_i} \epsilon)  \, \chi_{x_{t_i}^\epsilon\in D} \,
  \left \|  f(x_{t_i}^\epsilon, \cdot)-\bar  f(x_{t_i}^\epsilon)\right\|_{s}. 
  \end{aligned}$$
Summing up over $i$ and making a time change we obtain that
 $$\begin{aligned}
 \left| \EE\sum_{i=0}^{N-1} \int_{t_i}^{t_{i+1}}f\left(x_{t_i}^\epsilon, y_r^{x_{t_i }^\epsilon}\right) \,dr- \Delta t_i \bar f(x_{t_i}^\epsilon) \right|
&=\sum_{i=0}^{N-1}\EE\left| \epsilon \int_{t_i/\epsilon}^{t_{i+1}/\epsilon}f\left(x_{t_i}^\epsilon, z^{x_{t_i}^\epsilon}_{r}\right) \,dr- \Delta t_i \bar f(x_{t_i}^\epsilon) \right|\\
 &\le \bar \alpha  C\lambda(\f {\Delta t_i} \epsilon)  \sum_{i=0}^{N-1}\Delta t_i,
 \end{aligned}$$ 
and thus  conclude the proof.
\hfill$\square$
\end{proof}

For the application of the LLN, we must ensure the size of the sub-interval to be sufficiently large and  we should consider $\Delta t_i/\epsilon$ to be of order $\infty$ as $\epsilon\to 0$. Then we must ensure that $z_{\f t\epsilon}^{x_{t_i}^\epsilon}=y_r^{x_{t_i}^\epsilon}$  is an approximation for the fast variable $y_t^\epsilon$ on the sub-interval $[t_i, t_{i+1}]$. A crude counting shows that the distance of the 
two, beginning with the same initial value, is  bounded above by $\f {\Delta t_i} \epsilon$.
To obtain better estimates, we must choose the size of the interval carefully and use the slower evolutions of the slow variables on the sub-intervals and the Lipschitz continuity  of the driving vector fields $Y_i$.  We describe the intuitive idea for $\R^n\times \R^d$, assuming all vector fields are in $BC^\infty$. We use the Lipschitz continuity of the vector fields $\f 1 \epsilon Y_i$. On $[0, r]$, we have a pre-factor of $\f 1 \epsilon$ from  the stochastic integrals and $\f r \epsilon$ from the deterministic interval (by Holder's inequality). Then there exists a constant $C$ such that
$$\EE \left|y^\epsilon _{r}-y^{x_{t_i}^\epsilon} _{r}\right|^2
\le C (\f 1 \epsilon+ \f {\Delta t_i} {\epsilon^2}) \int_{t_i} ^r  \EE \left|x^\epsilon _{s }- x^\epsilon _{t_i }\right|^2 ds+
C (\f 1 \epsilon+ \f {\Delta t_i} {\epsilon^2}) \int_{t_i}^r \EE \left|y^\epsilon _{s }- y^{x_{t_i}^\epsilon} _{s }\right|^2 ds.$$
By Lemma \ref{lemma2-2},  $ \EE \left|x^\epsilon _{s }-x^\epsilon _{t_i }\right|^2 \le \tilde C \Delta t_i$
on $[t_i, t_{i+1}]$ where $\tilde C$ is a constant and so
$$\EE \left|y^\epsilon _{r}- y^{x_{t_i}^\epsilon} _{r}\right|^2 
\le C\tilde C\Delta t_i  (\f {\Delta t_i} \epsilon+ \f {(\Delta t_i)^2} {\epsilon^2})e^{ C (\f{\Delta t_i} \epsilon+ \f {(\Delta t_i)^2} {\epsilon^2})}.$$
  If we take $\Delta t_i$ to be of the order $ \epsilon| \ln \epsilon |^{a}$ for a suitable $a>0$,
then the above qunatity converges to zero uniformly in $r$ as $\epsilon\to 0$. See. e.g.
\cite{hasminskii68, Freidlin78,  Freidlin-Wentzell, Veretennikov}.
\medskip

 In the next lemma we give the statement and  the details of the computation under our standard assumptions. In particular we assume that the sectional curvature of $G$ is  bounded.
Let $C,c, c'$ denote constants.
\begin{lem}\label{freeze-lemma}
Let $0=t_0<t_1<\dots <t_N=T$ and $\epsilon\in (0, 1]$. 
Let $$\alpha_i^\epsilon(C):= C \left(\f {\Delta t_i} \epsilon+ \f {(\Delta t_i)^2} {\epsilon^2}\right)
e^{ C (\f {\Delta t_i} \epsilon+ \f {(\Delta t_i)^2} {\epsilon^2})} \sup_{s\in [t_i, t_{i+1}]}\EE  \rho^2\left(x^\epsilon _{s }, x^\epsilon _{t_i }\right).$$
\begin{enumerate}
\item Suppose Assumption \ref{assumeY}.  Then there exist constants $c$ and $C$ such that:
$$\EE \rho^2\left(y^\epsilon _r, y^{x_{t_i}^\epsilon}_r\right) 
\le  \alpha_i^\epsilon(C)+ c{\sqrt K}\left(\alpha_i^\epsilon(C)\right)^{\f 12} \f{ \Delta t_i }{\epsilon}  e^{c \f {\Delta t_i}{\epsilon}}
$$
where $K$ is the bound on the sectional curvature of $G$.
\item Suppose furthermore  that there exists a constant $c'$ such that
 $$\sup_{i=0,1, \dots N-1} 
 \sup_{s,t \in [t_i, t_{i+1}]} \sup_{\epsilon \in (0, 1]}\EE \rho^2(x_s^\epsilon, x_t^\epsilon) \le c'\,|t-s|.$$
Then there exists a constant $C>0$ such that  for every  $\epsilon \in (0, 1]$,
$$
\EE \rho^2\left(y^\epsilon _{r }, y^{x_{t_i}^\epsilon}_ {r }\right) 
\le C\sqrt{\Delta t_i} \left( \f {(\Delta t_i)^2} {\epsilon^2}+\f {(\Delta t_i)^3} {\epsilon^3}\right)^{\f 12}  
e^{ C (\f {\Delta t_i} \epsilon+ \f {(\Delta t_i)^2} {\epsilon^2})}, \quad \forall r\in [t_i, t_{i+1}], \forall i.
$$
In particular, if  $\Delta t_i$ is of the order $ \epsilon| \ln \epsilon |^{a}$ where $a>0$, then
 $\sup_i \sup_{r\in [t_i, t_{i+1}]}\EE \rho^2\left(y^\epsilon _{r }, y^{x_{t_i}^\epsilon}_ {r }\right) $ is of order $\epsilon^\delta$ where $\delta\in (0, \f 12)$.
\end{enumerate}
\end{lem}
\begin{proof}
Since the sectional curvature of $G$  is bounded above by $K$, its conjugate radius is bounded from below by $\f \pi {\sqrt K}$.
 Let us consider a distance function on $N$ that agrees
with the Riemannian distance, which we  denote by $\rho$, on the tubular neighbourhood of the diagonal of $N\times N$ with radius $\f \pi {2\sqrt K}$. More precisely let $\tau:=\tau^\epsilon$ be the first exit time when the distance between $y_r^\epsilon$ and $y_r^{x_{t_i}^\epsilon}$ is greater than or equal to $A=\f \pi {2\sqrt K}$. We use the identity
$$\EE  \rho^2\left(y^\epsilon _{r\wedge \tau }, y^{x_{t_i}^\epsilon} _{r\wedge \tau }\right)
=\EE\left[ \rho^2\left(y^\epsilon _{r }, y_r^{x_{t_i}^\epsilon}\right) \chi_{r<\tau} \right]+A^2\,P(\tau \le r),
$$
to obtain that
$$P(\tau \le r) \le \f 1 {A^2}  \EE\left[  \rho^2\left(y^\epsilon _{r\wedge \tau },  y^{x_{t_i}^\epsilon} _{r\wedge \tau }\right)\right].$$
Thus, $$\EE \rho^2\left(y^\epsilon _{r }, y^{x_{t_i}^\epsilon} _{r }\right) 
\le \EE\left[ \rho^2\left(y^\epsilon _{r }, y_r^{x_{t_i}^\epsilon}\right) \chi_{r< \tau} \right] 
+ \EE\left[ \rho^4\left(y^\epsilon _{r }, y_r^{x_{t_i}^\epsilon} \right)  \chi_{r\ge \tau} \right]^{\f 12} \sqrt{P(\tau\le r)}.$$
 By the earlier argument, it is sufficient
to estimate $\EE  \rho^2\left(y^\epsilon _{r\wedge \tau }, y^{x_{t_i}^\epsilon} _{r\wedge \tau }\right)$, and we will show that $ \EE  \rho^2\left(y^\epsilon _{r\wedge \tau }, y^{x_{t_i}^\epsilon}  _{r\wedge \tau }\right)$ converges to zero sufficiently fast as $\epsilon \to 0$ to compensate with the possible divergence from the  factor $ \left( \EE  \rho^4\left(y^\epsilon _{r }, y^{x_{t_i}^\epsilon}  _{r }\right) \right)^{\f 12} $.

On $\{r<\tau\}$, $x, y$ are not on each other's cut locus,  we may apply It\^o's  formula to the pair of stochastic processes $( y_r^\epsilon,y_r^{x_{t_i}^\epsilon})$  and obtain
$$\begin{aligned}
\left[\rho( y_r^\epsilon,y_r^{x_{t_i}^\epsilon})\right]^2
=&\int_{t_i}^r d \rho^2\left( \f 1 {\sqrt \epsilon}\sum_{k=1}^{m_2}  Y_k(x_{s}^\epsilon, y_s^\epsilon) \circ dW_s^k +\f 1{\epsilon} Y_0(x_{s}^\epsilon, y_s^\epsilon)\,ds\right)\\
&+\int_{t_i}^r d \rho^2\left( \f 1 {\sqrt\epsilon}\sum_{k=1}^{m_2}  Y_k(x_{t_i}^\epsilon, y_s^{x_{t_i}^\epsilon}) \circ dW_s^k +\f 1{\epsilon} Y_0(x_{t_i}^\epsilon, y_s^{x_{t_i}^\epsilon})\,ds\right).
\end{aligned}$$
Here  the notation $d$ in the first $d\rho^2$  refers to differentiation w.r.t. the first  variable, as a gradient we use $\nabla^{(1)}(\rho^2)$, and the $d$ in the second $d\rho^2$  is with respect to the second variable whose gradient is denoted by $ \nabla^{(2)} (\rho^2)$.
However $\nabla^{(1)} (\rho^2) (x,y)=-\paral\nabla^{(2)} (\rho^2) (x,y)$, where $\paral$ denotes the parallel translation of the relevant gradient vector  along the geodesic from $y$ to $x$.
In the following let us denote by $d \rho^2$ the differential of $\rho^2$ w.r.t to the first variable. 
Using the assumption that  each $Y_k$, $k=1,\dots, m_2$, has bounded first order derivative, and the fact that $\nabla \rho$ and $\nabla^2 \rho$ are bounded, the latter follows from the assumption that  the sectional curvature is bounded, we see:
$$\left | d \rho^2 (Y_k) \left(x_{s}^\epsilon, y_s^\epsilon \right) -d \rho^2(\paral Y_k) 
 \left(x_{t_i}^\epsilon, y_s^{x_{t_i}^\epsilon} \right)\right| 
 \le  2\rho( y_s^\epsilon,y_s^{x_{t_i}^\epsilon}) \left(\rho( x_s^\epsilon, x_{t_i}^\epsilon) + \rho( y_s^\epsilon,y_s^{x_{t_i}^\epsilon}) \right).  $$
It is useful to observe that $Y_i$ is a vector field on $G$ depending on $x\in N$, so the (product) distance function
 on $N\times G$ is needed for the estimate. On the other hand we only need to control 
 the Hessian of the Riemannian distance on $G$ and the assumption on the boundedness of the sectional curvature of $G$ suffices.

A similar estimate applies to the first order differential involving $\tilde Y_0$, the sum of the Stratnovich correction for the stochastic integrals and $Y_0$.
 Again we use the
assumption that each $Y_k$ where  $k$ ranges from $1$ to $m_2$  is bounded, and $\tilde Y_0$ has bounded first order covariant derivative. To summing up, for a constant $C$ independent of $\epsilon$ and $i$, we have
$$\begin{aligned}
&\EE \rho^2\left(y^\epsilon _{r\wedge \tau }, y^{x_{t_i}^\epsilon}_ {r\wedge \tau }\right) 
\le C (\f 1 \epsilon+ \f {\Delta t_i} {\epsilon^2})\left( \EE \int_{t_i} ^{r\wedge \tau }  \rho^2\left(x^\epsilon _{s }, x^\epsilon _{t_i }\right) ds+
\EE \int_{t_i}^{r\wedge \tau } \rho^2\left(y^\epsilon _{s }, y^\epsilon _{s }\right) ds\right).
\end{aligned}$$
Use  Gronwall's inequality we obtain that,
\begin{equation}
\EE \rho^2\left(y^\epsilon _{r\wedge \tau }, y^{x_{t_i}^\epsilon}_ {r\wedge \tau }\right) 
\le C \left(\f {\Delta t_i} \epsilon+ \f {(\Delta t_i)^2} {\epsilon^2}\right)
 \sup_{s\in [t_i, t_{i+1}]}
\EE  \rho^2\left(x^\epsilon _{s }, x^\epsilon _{t_i }\right)
e^{ C (\f {\Delta t_i} \epsilon+ \f {(\Delta t_i)^2} {\epsilon^2})}.
\end{equation}
We can now plug in  the uniform estimates that $\EE \rho^2\left(x^\epsilon _{s }, x^\epsilon _{t_i }\right)
\le C|t_i-s|$ we see that
$$\EE \rho^2\left(y^\epsilon _{r\wedge \tau }, y^{x_{t_i}^\epsilon}_ {r\wedge \tau }\right) 
\le C\Delta t_i \left(\f {\Delta t_i} \epsilon+ \f {(\Delta t_i)^2} {\epsilon^2}\right) e^{ C (\f {\Delta t_i} \epsilon+ \f {(\Delta t_i)^2} {\epsilon^2})}.
$$
Observe that the constant here is independent of $\epsilon, i$ and independent of $r\in [t_i, t_{i+1}]$.
A similar estimates hold for $\EE \rho^2\left(y^\epsilon _{r }, y^{x_{t_i}^\epsilon}_ {r }\right)\chi_{\tau>r}$.

On $\{r>\tau\}$ we use a more crude estimate, which we obtain without using estimates on the slow variables at time $s$ and time $t_i$.  It is sufficient to estimate 
$\EE \rho^4\left(y^\epsilon _{r }, y_{t_i}^\epsilon \right) $ and $\EE \rho^4\left( y^{x_{t_i}^\epsilon}_ {r },y_{t_i}^\epsilon\right) $. Observing that on $[t_i, t_{i+1}]$,
the  processes begin with the same initial point and the driving vector fields of the SDEs to which they are solutions are $\f 1 \epsilon Y_i(x_r^\epsilon, \cdot)$ and  $\f 1 \epsilon Y_i(x_{t_i}^\epsilon, \cdot)$ respectively. We have assumed that
$\sum_{k=1}^m|Y_k|$ and $\tilde Y_0$ are bounded.
We then apply Lemma \ref{lemma2-2} to these SDEs.  In Lemma \ref{lemma2-2}  we take $K= \f c \epsilon$ where $c$ is a constant.
  Then we have
\begin{equation}
\EE \rho^4\left( y^{x_{t_i}^\epsilon}_ {r },y_{t_i}^\epsilon\right) 
+\EE \rho^4\left(y^\epsilon _{r }, y_{t_i}^\epsilon \right)
 \le c \left( \Delta t_i + \f{ \Delta t_i }{\epsilon} \right) e^{c \f {\Delta t_i}{\epsilon}}.
\end{equation}
Again, the  constant  is independent of $\epsilon, i$ and independent of $r\in [t_i, t_{i+1}]$.
We put the two estimates together to see that
$$\begin{aligned}
\EE \rho^2\left(y^\epsilon _{r }, y^{x_{t_i}^\epsilon}_ {r }\right) 
\le &C\Delta t_i \left(\f {\Delta t_i} \epsilon+ \f {(\Delta t_i)^2} {\epsilon^2}\right) e^{ C (\f {\Delta t_i} \epsilon+ \f {(\Delta t_i)^2} {\epsilon^2})}\\
&
+\f {2\sqrt K}\pi
\left(C\Delta t_i \left(\f {\Delta t_i} \epsilon+ \f {(\Delta t_i)^2} {\epsilon^2}\right) e^{ C (\f {\Delta t_i} \epsilon+ \f {(\Delta t_i)^2} {\epsilon^2})}\right)^{\f 12}\sqrt {c }\left( \Delta t_i + \f{ \Delta t_i }{\epsilon} \right) ^{\f 12}e^{\f 12 c \f {\Delta t_i}{\epsilon}}.
\end{aligned}$$
For $\epsilon$ small the first term is small. The second factor in the second term on the right hand side is large. We conclude that for another constant $\tilde C$,
$$
\EE \rho^2\left(y^\epsilon _{r }, y^{x_{t_i}^\epsilon}_ {r }\right) 
\le \tilde C\sqrt{\Delta t_i} (1+\epsilon)^{\f 12}\left( \f {(\Delta t_i)^2} {\epsilon^2}+\f {(\Delta t_i)^3} {\epsilon^3}\right)^{\f 12}  
e^{\tilde C (\f {\Delta t_i} \epsilon+ \f {(\Delta t_i)^2} {\epsilon^2})}.
$$
Let us suppose that $\Delta t_i\sim\epsilon |\ln \epsilon|^{a}$.
Then  the exponent
 $\f {\Delta t_i} \epsilon+ \f {(\Delta t_i)^2} {\epsilon^2}\sim |\ln \epsilon|^{2a}$.
 So for a constant $C'$,
 $$\begin{aligned}
\EE \rho^2\left(y^\epsilon _{r }, y^{x_{t_i}^\epsilon}_ {r }\right) 
\le C' \sqrt\epsilon\, |\ln \epsilon|^{2a}\, e^{\tilde C |\ln \epsilon|^{2a}}.
\end{aligned}$$
 The right hand side  is of order $\epsilon^\delta$  for $\delta<\f 12$.
We conclude the proof.
\hfill$\square$
\end{proof}

The next lemma is on the convergence of Riemannian sums in the stochastic averaging procedure
and the continuity of stochastic averages of a function with respect to a family of measures $\mu_x$.

\begin{lem}\label{Riemann-sum}
Suppose that for a sequence of numbers $\epsilon_n\downarrow 0$,  $x_\cdot^{\epsilon_n}$ converges almost surely in $C([0,T]; N)$ to  a stochastic process $x_\cdot$.
Suppose that there exists a constant  $p\ge 1$ s.t.  for $|s-t|$ sufficiently small,
$$\EE\left[  \sup_{0\le r \le T}\rho^{2p}(x_r^\epsilon, O)\right]<\infty, \quad
\EE \rho(x_s^\epsilon, x_{t}^\epsilon)^2 \le C|t-s|, \quad \forall \epsilon\in (0, 1].$$ 
Let $\mu_x$ be a family of probability measures on $G$, continuous in $x$ in the total variation norm.
Let $f:N\times G\to \R$ be a $BC^1$ function. Let $0=t_0<t_1<\dots <t_N=T$ and let $C_1= |f|_\infty K_2+ |\nabla f|_\infty$. Then, 
 the following statements hold:
\begin{enumerate}
\item [(i)] $$\sup_{t\in [0,T]}\EE\left|\int_G f(x_{t}^{\epsilon_n}, z) \; \mu^{x_{t}^{\epsilon_n}}(dz)-
\int_G f(\bar x_t, z) \; \mu^{\bar x_t}(dz) \right| \to 0.$$
In particular, the following converges in $L^1$,
$$\left| \int_0^t \int_G f(x_s^{\epsilon_n}, z) \; \mu^{x_s^{\epsilon_n}}(dz) \,ds -
\int_0^t \int_G f(\bar x_s, z) \; \mu^{\bar x_s}(dz)\,ds\right|\to 0.$$
\item[(ii)]  The following convergence is uniform in $\epsilon$:
$$\EE \left| \sum_{i=0}^{N-1}\Delta t_i\int_G f(x_{t_i}^\epsilon, z) \; \mu^{x_{t_i}^\epsilon}(dz)-\int_0^T \int_G f(x_s^\epsilon) \mu^{x_s^\epsilon} (dz) ds\right| \to 0.$$
Consequently, the Riemannian sum $\sum_{i=0}^{N-1}\Delta t_i\int_G f(\bar x_{t_i}, z) \; \mu^{\bar x_{t_i}}(dz)$ converges in $L^1$ to
$ \int_0^T \int_G f(\bar x_s, z) \mu^{\bar x_s} (dz) ds$.
\end{enumerate}

\end{lem}

\begin{proof}
Suppose that $x_\cdot^{\epsilon_n}$ converges to $\bar x_\cdot$. We simplify the notation by assuming that $x^\epsilon\to x$ almost surely.
We may assume that $N$ is not compact, the compact case is easier.
Let $D_n$ be a family of relatively compact open set such that $D_n\subset B_{a_n}\subset B_{a_{n+2} }\subset D_{n+1}$ where $B_{a_n}$ is the geodesic ball centred at $O$ of radius $a_n$ where $a_n\to \infty$.  This exists by a theorem of Greene and Wu.
 For any $t\in [0,T]$ and for any $\epsilon\in (0, 1]$,
$$\begin{aligned}
&\left|\int_G f(x_{t}^\epsilon, z) \; \mu^{x_{t}^\epsilon}(dz)-
\int_G f(\bar x_t, z) \; \mu^{\bar x_t}(dz) \right| \\
&\le
 \int_G \left| f(x_t^\epsilon, z) - f(\bar x_t, z) \right| \; \mu^{\bar x_t}(dz)+
\left|  \int_G f(\bar x_t, z)  \mu^{\bar x_t}(dz)-  \int_G f(\bar x_t, z) \mu^{x_t^\epsilon}(dz) \right|\\
&\le|\nabla f|_\infty \;  \rho( x_t^\epsilon, \bar x_t)+ |f|_\infty  | \mu^{\bar x_t}-  \mu^{x_t^{\epsilon}}|_{TV}.
\end{aligned}$$

 We have control over $ \rho( x_t^\epsilon, \bar x_t)$, it is bounded by $ \rho( x_t^\epsilon, O)$ and
 $ \rho( \bar x_t, O)$. By the assumption, they are bounded in $L^p$, uniformly in $\epsilon\in (0,1]$ and in $t\in [0,T]$. Similarly
 we also have uniform control over $P(\bar x_t\not \in D_n)$ and $P( x_t^\epsilon \not \in D_n)$,
 they are bounded by $c\f 1 n $ where $c$ is a constant. We observe that 
$$  | \mu^{\bar x_t}-  \mu^{x_t^{\epsilon}}|_{TV}
\le  | \mu^{\bar x_t}-  \mu^{x_t^{\epsilon}}|_{TV}\chi_{\bar x_t\in D_n}\chi_{\bar x_t^\epsilon\in D_n}+2 ( \chi_{\bar x_t\not \in D_n}+\chi_{x_t^\epsilon\not \in D_n})$$
 and there exists $c_n$ such that  
 $ | \mu^{\bar x_t}-  \mu^{x_t^\epsilon}|_{TV}\chi_{\bar x_t\in D_n}\chi_{\bar x_t^\epsilon\in D_n}
 \le c_n\rho( x_t^{\epsilon_n}, \bar x_t)$.
We take $n$ large, so that $P(\bar x_t\not \in D_n)$ and $P( x_t^\epsilon \not \in D_n)$ are as small as we want. Then for $n$ fixed we see that 
the $c_n\rho( x_t^{\epsilon_n}, \bar x_t)$ converges, as $\epsilon\to 0$,  in $L^1$.
 Thus, $\sup_{0\le t\le T}\EE | \mu^{\bar x_t}-  \mu^{x_t^{\epsilon}}|\to 0$ and 
 $$\sup_{0\le t\le T}\EE\left|\int_G f(x_{t}^{\epsilon_n}, z) \; \mu^{x_{t}^{\epsilon_n}}(dz)-
\int_G f(\bar x_t, z) \; \mu^{\bar x_t}(dz) \right|$$ converges to zero. This proves  part (i).
  Since $\EE \rho(x_s^\epsilon, x_{t_i}^\epsilon)^2 \le C|t_{i+1}-t_i|$,
$$\begin{aligned}
&\EE \left| \sum_{i=0}^{N-1}\Delta t_i\int_G f(x_{t_i}^\epsilon, z) \; \mu^{x_{t_i}^\epsilon}(dz)-\int_0^T \int_G f(x_s^\epsilon) \mu^{x_s^\epsilon} (dz) ds\right| \\
&\le T\, |\nabla f|_\infty \sup_{s\in [t_i, t_{i+1})}\EE [ \rho(x^\epsilon_s, x^\epsilon_{t_i})]
+|f|_\infty \,T
\sup_{s\in [t_i, t_{i+1})} \EE \left[| \mu^{x^\epsilon_s}-  \mu^{x^\epsilon_{t_i}}|_{TV}\right]\to 0.
\end{aligned}$$
The convergence can be proved,  again by  breaking the total variation norm into two parts, in one part  the processes are in $D_n$, and in the other part they are not. Since $x_t^\epsilon$ converges to $x_t$ as a stochastic process on $[0, T]$, we  also have that
$\EE \rho(x_s, x_{t_i}) \le C|t_{i+1}-t_i|$.  We apply the same argument  to $\bar x_t$ to obtain that
$$\begin{aligned}
&\EE\left|\sum_{i=0}^{N-1}\Delta t_i\int_G f(\bar x_{t_i}, z) \; \mu^{\bar x_{t_i}}(dz)-\int_0^T \int_G f(\bar x_s) \mu^{\bar x_s} (dz) ds\right|\to 0.
\end{aligned}$$
This concludes the proof.
 \hfill$\square$
\end{proof}
Suppose we assume  furthermore that there exists a constant $K$ such that 
$$|\mu^{x_1}-\mu^{x_2}|_{TV} \le K( 1+\rho(x_1, O)
+ \rho(x_2,O)) \rho( x_1, x_2).$$
Then explicit estimates can be made for the convergence in Lemma \ref{Riemann-sum}, e.g. $$\begin{aligned}
&\left|\int_G f(x_{t}^\epsilon, z) \; \mu^{x_{t}^\epsilon}(dz)-
\int_G f(\bar x_t, z) \; \mu^{\bar x_t}(dz) \right| \\
&\le|\nabla f|_\infty \;  \rho( x_t^\epsilon, \bar x_t) +|f|_\infty  K( 1+\rho^p( \bar x_t, O)
+ \rho^p(x_t^\epsilon,O)) \rho( x_t^\epsilon, \bar x_t).
\end{aligned}$$
To this we may apply H\"older's inequality and obtain:
$$\begin{aligned}
&\EE\left|\int_0^T \int_G f(x_{t}^{\epsilon_n}, z) \; \mu^{x_{t}^{\epsilon_n}}(dz)\, dt-
\int_0^T \int_G f(\bar x_t, z) \; \mu^{\bar x_t}(dz)\, dt \right| \\
& \le   |\nabla f|_\infty \;   \EE\int_0^T\rho( x_t^\epsilon, \bar x_t) \; dt
 +|f|_\infty  K \EE \left| \int_0^T ( 1+\rho^p( \bar x_t, O)
+ \rho^p(x_t^\epsilon,O)) \rho( x_t^\epsilon, \bar x_t) \, dt\right|\\
& \le   |\nabla f|_\infty \;  T\EE \sup_{s\le t}\rho( x_t^\epsilon, \bar x_t) 
 +|f|_\infty  K T  \sqrt{\EE\sup_{t\le T}  ( 1+\rho^p( \bar x_t, O)
+ \rho^p(x_t^\epsilon,O))^2} \sqrt{  \EE \sup_{t\le T}  \rho^2( x_t^\epsilon, \bar x_t) }.
\end{aligned}
$$
%From this and other global estimates one expects to obtain  an estimate on the rate of the convergence of $x_t^\epsilon$, which is beyond the scope of the paper.

 In the proposition below we are interested in the time average concerning a product function $f_1f_2$, where $f_1:N\to \R$ is $C^\infty$ with has compact support and $f_2:G\to \R$ is smooth.
\begin{prop}\label{averaging-proposition}
Suppose the following conditions.
\begin{enumerate}
\item [(1)]  $\mu_x$ is a family of probability measures on $G$ for which
the locally uniform LLN  assumption (Assumption \ref{assumptions-Birkhoff}) holds. \item [(2)] Assumption \ref{assumeY}.
\item[(3)] There exist constants $p\ge 1$ and $c$ such that for  $s,t\in [r_1, r_2]$ where $r_2-r_1$ is sufficiently small,
 $$\sup_{\epsilon \in (0, 1]}\sup_{s,t\in [r_1, r_2]}\EE \rho^2(x_s^\epsilon, x_t^\epsilon) \le c|t-s|,
\quad  \sup_{0\le s\le T} \sup_{\epsilon \in (0, 1]}\EE \rho^{2p}(x_s^\epsilon, O)<\infty.$$
\item[(4)]  $\epsilon_n$ is a sequence of numbers converging to $0$ with $\sup_{t\le T} \rho(x_t^{\epsilon_n},\bar x_t)$ converges to zero almost surely. 
\item [(5)]Let $f:N\times G\to \R$ be a smooth and globally Lipschitz continuous function. Suppose that either $f$ is independent of the first variable or  for each $y\in G$,   the support of $f(\cdot, y)$ 
is contained in a compact set $D$.
\end{enumerate}
Then the following random variables converge to zero in $L^1$:
$$\int_0^T f(x_s^\epsilon, y_s^\epsilon) ds-\int_0^T\int_G f(\bar x_s, z)\, \mu^{\bar x_s} (dz)\, ds.  $$
\end{prop}
\begin{proof}
 Let $0=t_0<t_1<\dots<t_N=T$ and let $\Delta t_i=t_{i+1}-t_i$. Then, recalling the notation given in (\ref{time-change}),
$$\begin{aligned}
&\int_0^T f(x_s^\epsilon, y_s^\epsilon) ds
=\sum_{n=0}^{N-1} \int_{t_i}^{t_{i+1}} f(x_s^\epsilon, y_s^\epsilon) ds\\
=&\sum_{n=0}^{N-1} \int_{t_i}^{t_{i+1}} \left[ f(x_s^\epsilon, y_s^\epsilon) -f(x_{t_i}^\epsilon, y_s^\epsilon) \right]\,ds
+\sum_{n=0}^{N-1} \int_{t_i}^{t_{i+1}} \left[ f(x_{t_i}^\epsilon, y_s^\epsilon)
-f\left(x_{t_i}^\epsilon, y^{x_{t_i }^\epsilon}_r\right) \right]\,ds\\
&+\sum_{n=0}^{N-1}\left[  \int_{t_i}^{t_{i+1}}f\left(x_{t_i}^\epsilon, y^{x_{t_i }^\epsilon}_r)\right) \,ds- \Delta t_i \int_G f(x_{t_i}^\epsilon, z) \mu^{x_{t_i}^\epsilon} (dz)\right]\\
&+\left[ \sum_{n=0}^{N-1} \Delta t_i\, \int_G f(x_{t_i}^\epsilon, z) \mu^{x_{t_i}^\epsilon} (dz)
-\int_0^T \int_G f( x_s^\epsilon,z) \mu^{ x^\epsilon_s} (dz)\,ds\right]\\
&+\left[\int_0^T \int_G f( x_s^\epsilon,z) \mu^{ x^\epsilon_s} (dz)\,ds
-\int_0^T \int_G f(\bar x_s,z) \mu^{\bar x_s} (dz)\,ds\right]
+\int_0^T \int_G f(\bar x_s,z) \mu^{\bar x_s} (dz)\,ds.
\end{aligned}$$
Using the fact that $f$ is Lipschitz continuous  in the first variable and the assumptions
on the moments of  $\rho(x_t^\epsilon, x_s^\epsilon)$ we see that for a constant $K$, 
$$\begin{aligned}
\sum_{i=0}^{N-1} \EE\int_{t_i}^{t_{i+1} } \left| f(x_r^\epsilon, y_r^\epsilon)-f(x_{t_i}^\epsilon, y_r^\epsilon)\right|\,dr &\le  K\sum_{i=0}^{N-1} \int_{t_i}^{t_{i+1} }
 \EE  \rho (x_r^\epsilon, x_{t_i}^\epsilon)\, dr\\
&\le K  \;T \EE  \rho (x_r^\epsilon, x_{t_i}^\epsilon) \le TKc\max_i \sqrt{\Delta t_i} .
\end{aligned}$$
By choosing $\Delta t_i=o(\epsilon)$ we see that the first term on the right hand side converges to zero.
The converges of the second term follows directly from Lemma \ref{freeze-lemma} by choosing
$\Delta t_i\sim \epsilon|\ln \epsilon|^a$ where $a>0$ and Assumption \ref{assumeY}.
 By Lemma \ref{lemma2-3-2} and Assumption \ref{assumptions-Birkhoff}, the third term converges 
 if we choose $\f {\epsilon}{\Delta t_i}=o(\epsilon)$. The convergence of the fourth and fifth terms follow respectively from part (i) and part (ii) of Lemma \ref{Riemann-sum}.
 \hfill$\square$
\end{proof}
 
  We are now ready to prove the main averaging theorem, Theorem 2 in  section~\ref{results}.
  This proof has the advantage for being concrete, from this an estimate for thew  rate of convergence is also  expected. 
 \begin{thm}\label{main-2}
Suppose  the following statements hold.
\begin{enumerate}
\item [(a)]
 Assumptions \ref{assumeX} and \ref{assumeY}. 
\item [(b)] There exists a family of  probability measures $\mu_x$ on $G$ for which the locally uniform LLN  assumption (Assumption \ref{assumptions-Birkhoff}) holds.  
\end{enumerate}
Then, as $\epsilon\to 0$,    the family of stochastic processes $\{x_t^\epsilon, \epsilon>0\}$ converges weakly
on any compact time intervals  to a Markov process with generator $\bar \L$.
\end{thm}
\begin{proof}
By Prohorov's theorem, a set of probability measures is tight if and only if its relatively weakly compact,  i.e. every sequence has a sub-sequence that converges weakly.
It is therefore sufficient to prove that every limit process of the stochastic processes $x_t^\epsilon$ is a Markov process with the same Markov generator. 
Every sequence of weakly convergent stochastic processes on an interval $[0,T]$
can be realised on a probability space as a sequence of stochastic processes that converge almost surely  on $[0,T]$ with respect to the supremum norm in time. It is sufficient to prove that if  a subsequence $\{x_t^{\epsilon_n}\} $ converges almost surely on $[0,T]$,  the limit is a Markov process with generator $\bar \L$. For this we apply Stroock-Varadhan's martingale method \cite{Stroock-Varadhan, Papanicolaou-Stroock-Varadhan77}. To ease notation we may assume that the whole family $x_t^\epsilon$ converges almost surely.  Let  $f$ be a real valued smooth function on $N$ with compact support. Let $\bar x_t$ be the limit Markov process. We must prove that
$f(\bar x_t) -f(x_0)-\int_0^t \bar  \L f(\bar x_r)dr $
is a martingale. In other words we prove that  for any bounded measurable random variable $G_s\in \F_s$ and for any $s<t$,  $\EE\left( G_s (f(\bar x_t)-f(\bar x_s) -\int_s^t \bar \L f(x_r)dr)\right)=0$.
On the other hand, for each $\epsilon>0$,
$$f(x_t^\epsilon)-f(x_s^\epsilon)-\int_s^t \left( \f 1 2\sum_{i=1}^{m_1} (X_i(\cdot, y_r^\epsilon))^2 f+X_0(\cdot, y_r^\epsilon )f \right)   (x_r^\epsilon) \;dr$$
is a martingale. Let us  introduce the notation:
$$F(x_r^\epsilon, y_r^\epsilon)= \left( \f 1 2\sum_{i=1}^{m_1} (X_i(\cdot, y_r^\epsilon))^2f +X_0(\cdot, y_r^\epsilon ) f\right)   (x_r^\epsilon).$$
Since $x_t^\epsilon$ converges to $\bar x_s$ it is sufficient to prove that as $\epsilon\to 0$,
$$\EE \left[ G_s \left( \int_s^t F(x_r^\epsilon, y_r^\epsilon) dr-
\int_s^t \bar \L f(x_r)dr\right) \right]\to 0.$$
Even simpler we only need to  prove that $\int_s^t F(x_r^\epsilon, y_r^\epsilon)dr$ converges to
$\int_s^t \bar \L f(x_r)dr $ in $L^1$. 
Under Assumption \ref{assumeX}, we may apply  Lemma \ref{lemma4.1}  from which we see that 
conditions (3) and (4) of  Proposition \ref{averaging-proposition} hold.  Since $f$ has compact support, $F$ has compact support in the first variable. 
We may apply Proposition \ref{averaging-proposition} to the function $F$ to complete the proof.
\hfill$\square$
\end{proof}

We remark that  the locally uniform law of large numbers hold  if $G$ is compact,  if $\L_x$ satisfies {\it strong H\"ormander's condition}, or if $\L_x$ satisfies  H\"ormander's condition with the additional assumption that $\L_x$ has a unique invariant probability measure. 

We obtain the following Corollary.

\begin{corollary}
Let $G$ be compact. Suppose  Assumptions \ref{assumeX} and \ref{assumeY}. 
Suppose that $\L_x$ satisfies  H\"ormander's condition  and  that it  has a unique invariant probability measure.
Then  $\{x_t^\epsilon, \epsilon>0\}$ converges weakly,  on any compact time intervals, to a Markov process with generator $\bar \L$.
\end{corollary}
From the proof of Theorem \ref{main-2}, the Markov generator $\bar \L$ given below.
\begin{equation}
\label{limit-op}
\bar \L  f(x)=\int_G \left(\f 12 \sum_{i=1}^{m_1} X_i^2(\cdot, y) f+ X_0 (\cdot, y) f\right) (x)\, \mu_x(dy).
\end{equation}

\subsection*{Appendix}\label{appendix}
It is possible to write the operator $\bar \L$ given by (\ref{limit-op})  as a sum of squares of vector fields. For this  we need  an auxiliary family of vector fields $\{E_1, \dots, E_{n_1}\}$ with the property that at each point $x$ they span the tangent space $T_xN$. 
 Let us write each vector field $X_i(\cdot, y)$ in this basis and denote by $X_i^k(\cdot,y)$ its coordinate functions, so
$X_i(x,y)=\sum_{k=1}^{n_1} X_i^k(x,y) E_k(x)$.
Set $$\begin{aligned}
a_{k,l}(x,y) &=  \sum_{i=1}^{m_1} X_i^k(x,y)X_i^l(x,y), \\
b_0^k(x,y)&=\f 12 \sum_{l=1}^{n_1}  \sum_{i=1}^{m_1} X^l_i (x,y) \left(\nabla X_k^l(\cdot, y)) (E_l(x))\right)+ X_0^k(x,y),
\end{aligned}$$
where $\nabla$ denotes differentiation with respect to the first variable.We observe that $$  \f 1 2 \sum_{i=1}^{m_1}(X_i(\cdot, y)^2 f) (x) +(X_0(\cdot, y)f)(x)
=  \f 1 2 \sum_{k,l=1}^{n_1} a_{k,l}(x,y) (E_kE_lf) (x)
+\sum_{k=1}^{n_1}b_0^k(x,y) (E_k f)(x). $$

If $\mu_x$ is a family of probability measures on $G$, we set
\begin{equation}
\begin{aligned}
\bar \L=  \f 1 2  \sum_{k,l=1}^{n_1}\left(\int_Ga_{k,l}(x,y)   \mu_x(dy)  \right) \, E_kE_l
+\sum_{k=1}^{n_1}\left( \int_G b_0^k(x,y) dy\right) \;E_k .
\end{aligned} 
\end{equation}

The auxiliary vector fields can be easily constructed. For example, we may use the 
gradient vector fields coming from an isometric embedding $i: N\to \R^{n_1}$. 
Then they have the 
following properties. For $e\in \R^{n_1}$, we define $E(x)(e)=\sum_{i=1}^{n_1} E_i(x) e_i$ where $\{e_i\}$ is an orthonormal basis 
of $\R^{n_1}$. Then $\R^{n_1}$ has a splitting of the form $\ker[X(x)]^\perp\oplus X(x)$
and $X(e)$ has vanishing derivative for $e\in \ker[X(x)]^\perp$. 
We may also use a `moving frames' instead of the gradient vector fields. This is particularly useful if 
 $N$ is an Euclidean space,  or a compact space, or  a Lie group.  For such spaces and their moving frames, the assumption that  $X_1, \dots, X_k$ and their two  order derivatives, $X_0$ and $\nabla X_0$ are bounded can be expressed by the boundedness of the functions $a_{k,l}$ and $b_0^k$
and their derivatives.

\section{Re-visit  the examples}
\label{examples2}

\subsection{A dynamical description for hypo-elliptic diffusions}\label{dynamical2} 
Let us consider two further generalisations to the dynamical theory for Brownian motions described in \S \ref{dynamical}. Both cases allow degeneracy in the fast variables.
One of which has the same type of reduced random ODE and  is closer to Theorem 2A.  We state this one first and will take $M$ compact for simplicity.
\begin{prop} Let $M$ be compact.
Suppose that in (\ref{ou-1}),  we replace the orthonormal basis $\{A_1, \dots, A_N\}$  and $A_0$
by a vectors  $\{ A_1, \dots A_{m_2}\}\subset \so(n)$ with the property that these vectors together with their commutators generates $\so(n)$. (Take $A_0=0$ for simplicity. Then,  as $\epsilon\to 0$, the rescaled  position stochastic processes, $x_{\f t \epsilon}^\epsilon$,  converges to a scaled Brownian motion. Their horizontal lifts from $u_0$ converge also. \end{prop}
\begin{prop}
The scale is determined by the eigenvalues of the symmetry matrix $\sum_{i=1}^{m_2} (A_i)^2$.
\end{prop}
\begin{proof}
The reduced equation is as before:
  \begin{equation*}\left\{\begin{aligned}
& \f {d} {dt} \tilde x_t^\epsilon=H_{\tilde x_t^\epsilon} (g_{\f t \epsilon} e_0),
 \quad  \tilde x_0^\epsilon=u_0,\\
&dg_t=  \sum_{k=1}^{m_2}{g_t} A_k \circ dw_t^k+g_t A_0\, dt, \quad g_0=Id. 
\end{aligned}\right.
\end{equation*}
We observe that the operator $\sum_{i=1}^{m_2} (A_i^*)^2+A_0^*$ satisfies H\"ormander's condition and has a unique invariant probability measure. It is symmetric w.r.t the bi-invariant Haar measure $dg$, and the only invariant measure is $dg$.
Then we apply Theorem 6.4 from \cite{Li-limits} to conclude.
\end{proof}

Suppose instead we consider the following SDE, in which the horizontal part involves a stochastic integral
\begin{equation}\label{ou-2}\left\{
\begin{split}du_t^\epsilon&= H_{u_t^\epsilon}(e_0)\,dt+\sum_{j=1}^{m_1} H(u_t^\epsilon)(e_j)\circ dB_t^j
+{1\over \sqrt \epsilon} \sum_{k=1}^{m_2}A_k^*(u_t^\epsilon)\circ dW_t^k+ A_0^*
(u_t^\epsilon) \,dt,
 \\ u_0^\epsilon&=u_0.
\end{split}\right.
\end{equation}
where $e_j\in \R^n$.

\begin{prop}
Suppose that $M$ has bounded sectional curvature.
Suppose that $\{A_0, A_1, \dots, A_{m_2}\}$ and their iterated  brackets (commutators) generate the vector space $\so(n)$.  Suppose that $\{e_1, \dots, e_{m_1}\}$
is an orthonormal set.
Then as $\epsilon\to 0$, the  position component of $u_t^\epsilon $, $x_t^\epsilon$,   converges to a rescaled Brownian motion, scaled by $\f {m_1} n$ where $n=\dim(M)$. Their horizontal lifts  converge also to a  horizontal Brownian motion with the same scale.
\end{prop}

\begin{proof}
Set $x_t^\epsilon=\pi(u_t^\epsilon)$, where $\pi$ takes an frame to its base point. Then  $x_t^\epsilon$ is the position process. Then
$$dx_t^\epsilon=\sum_{i=1}^{m_1} u_t^\epsilon (e_i) \circ dB_t^i+u_t^\epsilon e_0\,dt.$$
Let $\tilde x_t^\epsilon$ denote the stochastic horizontal lifts of $x_t^\epsilon$.  Then from the nature of the horizontal vector fields and the horizontal lifts, this procedure introduces a twist to the Euclidean vectors $e_i$. If $g_t$ solves:
$$dg_t=  \sum_{k=1}^{m_2}{g_t} A_k \circ dw_t^k+g_t A_0\, dt$$
with initial value the identity, then
  $x_t^\epsilon$
satisfies the equation
$$d \tilde x_t^\epsilon=\h_{\tilde x_t^\epsilon} dx_t^\epsilon
=H(\tilde x_t^\epsilon) (g_{\f t \epsilon} e_0)dt+ \sum_{i=1}^{m_1} H(\tilde x_t^\epsilon) (g_\f{ t}{\epsilon} e_i) \circ dB_t^i.$$

Since $g_t$ does not depend on the slow variable, the conditions of  the Theorem is satisfied
provided $M$ has bounded sectional curvature. 

 The limiting process, in this case, will not be a fixed point. It is a Markov process on the orthonormal frame bundle with generator
$$\bar \L f(u)=\sum_{i=1}^{m_1}\int_{SO(n)} \nabla df (H(u)(ge_i), H(u)(ge_i)  )dg,$$
where $\nabla $ is a flat connection,   $\nabla  H_i$ vanishes. Let $dg$ denote the normalised bi-invariant Haar measure. Using this connection and an
orthonormal basis $\{e_i\}$ of $\R^n$, extending our orthonormal  set $\{e_1, \dots, e_{m_1}\}$  we see that
$$\bar \L f(u)=\sum_{k,l=1}^n \nabla df (H(E_k), H(E_l)) 
 \sum_{i=1}^{m_1}  \int_{SO(n)} \<e_k, ge_i\> \<  ge_i , e_l\> dg.$$
 It is easy to see that $$ \sum_{i=1}^{m_1}  \int_{SO(n)} \<e_k, ge_i\> \<  ge_i , e_l\> dg=
  \sum_{i=1}^{m_1} \delta_{k,l} \int_{SO(n)} \<e_k, ge_i\> ^2 dg=\f {m_1}n \delta_{k,l}.$$
  This means that  $$\bar \L f(u)=\f {m_1}n \sum_{k,l=1}^n \nabla df (H(e_k), H(e_k)).$$
  Thus the $\bar \L$ diffusion has Markov generator $\f 1 2 \f {m_1}n\Delta^H$
  where $\Delta^H$ is the horizontal diffusion and which means that $\pi(u_t^\epsilon)$ converges to a scaled Brownian motion as we have guessed.
\hfill $\square$
\end{proof}
\begin{problem}
The vertical vector fields  in (\ref{ou-2}) are left  invariant. 
Instead of left invariant vertical vector fields we may take more general vector fields and consider the following SDEs.
Let $f:OM\to \R$ be smooth functions, $e_j\in \R^n$ are unit vectors.
 Let us consider the equation \begin{equation}\label{ou-3}\left\{
\begin{split}du_t^\epsilon&= H_{u_t^\epsilon}(e_0)\,dt+\sum_{j=1}^{m_1} H(u_t^\epsilon)(e_j) \circ dB_t^i
+{1\over \sqrt \epsilon} \sum_{k=1}^{m_2} f_k(u_t^\epsilon)A_k^*(u_t^\epsilon)\circ dW_t^k+ A_0^*
(u_t^\epsilon) \,dt,
 \\ u_0^\epsilon&=u_0.
\end{split}\right.
\end{equation}
Then the horizontal lift of its position processes will, in general, depend on the slow variables.
It  would be interesting to determine explicit conditions on $f_k$ for which the averaging procedure  
is valid and  if so what is  the effective limit?
\end{problem}

\subsection{Inhomogeneous scaling of Riemannian metrics}\label{scaling2}
Returning to section \ref{scaling} we pose the following problem.

\begin{problem} With Theorem \ref{main-2}, we can now study a fully coupled system:
$$dg_t^\epsilon ={1\over \sqrt \epsilon} \sum_{k=1}^{m_2} (a_k  A_k)(g_t^\epsilon) \circ dB_t^k +{1\over \epsilon} (a_0A_0)(g_0^\epsilon)dt
+(b_0Y_0)(g_t^\epsilon) dt+\sum_{k=1}^{m_1} (b_k  Y_k)(g_t^\epsilon) \circ dW_t^k,$$ 
where $a_k, b_k$ are smooth functions. It would be interesting to study the convergence of the slow variables,  vanishing of the averaged processes, and  the nature of the  limits in terms of $a_k$ and $b_k$.
\end{problem}

\subsection {An averaging principle on principal bundles}\label{averaging2}
We return to the example in section \ref{averaging}.
 In the following proposition, $\nabla$ denotes the flat connection on the principal bundle $P$.
\begin{prop}
\label{level-thm}
Let $G$ be a compact Lie group and $dg$ its Haar measure. 
Assume that $M$  has bounded sectional curvature. Suppose that $\L_u$ satisfies H\"ormander's condition and has a unique invariant probability measure.  Suppose that $\theta_k^j$  are bounded with bounded derivatives. Define $$\begin{aligned}
a_{i,j}(u)&=\int_{G} \sum_{l=1}^{m_1} \<X_l(u,g), H_i(u ) \>\<X_l(ug), H_j(u ) \>\;dg,\\
b(u)&=   \int_{G} \left({1\over 2} \sum_{l=1}^{m_1}  \nabla_{X_l} X_l(ug)+X_0(ug)\right)dg.
\end{aligned}$$
 Then $\tilde x_{t\over\epsilon}^\epsilon$ converges weakly to a Markov process on $P$ with the Markov generator 
 $$\bar \L f(u) = df(b(u))
   			+{1\over 2}\sum_{i,j=1}^n a_{i,j}(u) \nabla df(H_i(u),H_j(u) ).$$
 \end{prop}

\begin{proof}
The convergence is a trivial consequence of Theorem \ref{main-2}. To identify the limit  let $f: P\to \R$ 
be any smooth function with compact support. Then $$
\begin{aligned}
f(\tilde x_t^\epsilon)&=f(g_0)+\,\sum_{l=1}^{m_1}\int_0^t df\left( X_l((\tilde x_s^\epsilon g_s^\epsilon) \right)dB_s^l
+\sum_{l=1}^{m_1}
\int_0^t \nabla df\left(X_l(\tilde x_s^\epsilon g_s^\epsilon),X_l(\tilde x_s^\epsilon g_s^\epsilon) \right)ds\\
&+ \sum_{l=1}^{m_1}  \int_0^t  df\left( \nabla_{X_l} X_l(\tilde x_s^\epsilon g_s^\epsilon)+X_0(\tilde x _s^\epsilon g_s^\epsilon) \right) ds.
\end{aligned}$$
Finally we take coordinates of $X_l$ w.r.t  the parallel  vector fields  $H_i$, c.f. the Appendix of \S\ref{proof}, to complete the proof.
\end{proof}

{\bf Conclusions and Other Open Questions. } In conclusion, the examples we studied treat some of the  simplest and yet universal  models, they can be studied using the method we have just developed. Even for these simple models many questions remain to be answered, including the questions stated in \S\ref{results}, \S\ref{dynamical} and \S\ref{scaling}. For example  we do not know the geometric nature of the limiting object. Concerning Theorem \ref{main-2},  we expect the conditions of the theorem improved  for more specific examples of manifolds, and expect an upper bound for the  rate of convergence if  the resolvents of the operators $\L_x$ is bounded in $x$ and if the rank of the operators and their quadratic forms are bounded, and also expect an averaging principle for
slow-fast SDEs driven by L\'evy processes, c.f. \cite{Hogele-Ruffino}.
 
\bibliographystyle{alpha}
\bibliography{Hopf.bib}

\newcommand{\etalchar}[1]{$^{#1}$}
\def\dbar{\leavevmode\hbox to 0pt{\hskip.2ex \accent"16\hss}d} \def\cprime{$'$}
  \def\cprime{$'$} \def\cprime{$'$} \def\cprime{$'$} \def\cprime{$'$}
  \def\cprime{$'$} \def\cprime{$'$} \def\cprime{$'$} \def\cprime{$'$}
  \def\cprime{$'$}
\begin{thebibliography}{BHVW17}

\bibitem[ABT15]{Angst-Bailleul-Tardif}
J\"urgen Angst, Isma\"el Bailleul, and Camille Tardif.
\newblock Kinetic {B}rownian motion on {R}iemannian manifolds.
\newblock {\em Electron. J. Probab.}, 20:no. 110, 40, 2015.

\bibitem[ADK08]{Albeverio-Daletskii-Kalyuzhnyi}
Sergio Albeverio, Alexei Daletskii, and Alexander Kalyuzhnyi.
\newblock Random {W}itten {L}aplacians: traces of semigroups, {$L^2$}-{B}etti
  numbers and index.
\newblock {\em J. Eur. Math. Soc. (JEMS)}, 10(3):571--599, 2008.

\bibitem[Arn89]{Arnold89}
V.~I. Arnol{\cprime}d.
\newblock {\em Mathematical methods of classical mechanics}.
\newblock Springer-Verlag, New York, 1989.

\bibitem[Arn93]{Arnaudon-homogeneous}
Marc Arnaudon.
\newblock Semi-martingales dans les espaces homog\`enes.
\newblock {\em Ann. Inst. H. Poincar\'e Probab. Statist.}, 29(2):269--288,
  1993.

\bibitem[Ati69]{Atiyah}
M.~F. Atiyah.
\newblock Algebraic topology and operators in {H}ilbert space.
\newblock In {\em Lectures in {M}odern {A}nalysis and {A}pplications. {I}},
  pages 101--121. Springer, Berlin, 1969.

\bibitem[Bau04]{Baudoin}
Fabrice Baudoin.
\newblock {\em An introduction to the geometry of stochastic flows}.
\newblock Imperial College Press, London, 2004.

\bibitem[BBB82]{Berard-Bergery-Bourguignon}
Lionel B{\'e}rard-Bergery and Jean-Pierre Bourguignon.
\newblock Laplacians and {R}iemannian submersions with totally geodesic fibres.
\newblock {\em Illinois J. Math.}, 26(2):181--200, 1982.

\bibitem[BC16]{Bally-Caramellino}
Vlad Bally and Lucia Caramellino.
\newblock Asymptotic development for the {CLT} in total variation distance.
\newblock {\em Bernoulli}, 22(4):2442--2485, 2016.

\bibitem[BF95]{Borodin-Freidlin}
A.~N. Borodin and M.~I. Freidlin.
\newblock Fast oscillating random perturbations of dynamical systems with
  conservation laws.
\newblock {\em Ann. Inst. H. Poincar\'e Probab. Statist.}, 31(3):485--525,
  1995.

\bibitem[BGL14]{Bakry-Gentil-Ledoux}
Dominique Bakry, Ivan Gentil, and Michel Ledoux.
\newblock {\em Analysis and geometry of {M}arkov diffusion operators}, volume
  348 of {\em Grundlehren der Mathematischen Wissenschaften [Fundamental
  Principles of Mathematical Sciences]}.
\newblock Springer, Cham, 2014.

\bibitem[BGV92]{Berline-Getzler-Vergne}
Nicole Berline, Ezra Getzler, and Mich{\`e}le Vergne.
\newblock {\em Heat kernels and {D}irac operators}.
\newblock Springer-Verlag, 1992.

\bibitem[BHVW17]{Birrell-Hottovy-Volpe-Wehr}
Jeremiah Birrell, Scott Hottovy, Giovanni Volpe, and Jan Wehr.
\newblock Small mass limit of a {L}angevin equation on a manifold.
\newblock {\em Ann. Henri Poincar\'e}, 18(2):707--755, 2017.

\bibitem[Bis08]{Bismut-Lie-group}
Jean-Michel Bismut.
\newblock The hypoelliptic {L}aplacian on a compact {L}ie group.
\newblock {\em J. Funct. Anal.}, 255(9):2190--2232, 2008.

\bibitem[BL05]{Bismut-Lebeau}
Jean-Michel Bismut and Gilles Lebeau.
\newblock Laplacien hypoelliptique et torsion analytique.
\newblock {\em C. R. Math. Acad. Sci. Paris}, 341(2):113--118, 2005.

\bibitem[Bor77]{Borodin77}
A.~N. Borodin.
\newblock A limit theorem for the solutions of differential equations with a
  random right-hand side.
\newblock {\em Teor. Verojatnost. i Primenen.}, 22(3):498--512, 1977.

\bibitem[BvR14]{Barret-vonRenesse}
Florent Barret and Max von Renesse.
\newblock Averaging principle for diffusion processes via {D}irichlet forms.
\newblock {\em Potential Anal.}, 41(4):1033--1063, 2014.

\bibitem[CCG{\etalchar{+}}10]{Chow-3}
Bennett Chow, Sun-Chin Chu, David Glickenstein, Christine Guenther, James
  Isenberg, Tom Ivey, Dan Knopf, Peng Lu, Feng Luo, and Lei Ni.
\newblock {\em The {R}icci flow: techniques and applications. {P}art {III}.
  {G}eometric-analytic aspects}, volume 163 of {\em Mathematical Surveys and
  Monographs}.
\newblock American Mathematical Society, Providence, RI, 2010.

\bibitem[CF10]{Cass-Friz}
Thomas Cass and Peter Friz.
\newblock Densities for rough differential equations under {H}\"ormander's
  condition.
\newblock {\em Ann. of Math. (2)}, 171(3):2115--2141, 2010.

\bibitem[CG16]{Catellier-Gubinelli}
R.~Catellier and M.~Gubinelli.
\newblock Averaging along irregular curves and regularisation of {ODE}s.
\newblock {\em Stochastic Process. Appl.}, 126(8):2323--2366, 2016.

\bibitem[CO16]{Crisan-Ottobre}
D.~Crisan and M.~Ottobre.
\newblock Pointwise gradient bounds for degenerate semigroups (of {UFG} type).
\newblock {\em Proc. A.}, 472(2195):20160442, 23, 2016.

\bibitem[DKK04]{Dolgopyat-Kaloshin-Koralov}
Dmitry Dolgopyat, Vadim Kaloshin, and Leonid Koralov.
\newblock Sample path properties of the stochastic flows.
\newblock {\em Ann. Probab.}, 32(1A):1--27, 2004.

\bibitem[DLPS17]{Duong-Lamacz-Peletier-Sarma}
Manh~Hong Duong, Agnes Lamacz, Mark~A. Peletier, and Upanshu Sharma.
\newblock Variational approach to coarse-graining of generalized gradient
  flows.
\newblock Preprint, 2017.

\bibitem[Dow80]{Dowell}
R.~M. Dowell.
\newblock {\em Differentiable Approximations to Brownian Motion on Manifolds}.
\newblock PhD thesis, University of Warwick, 1980.

\bibitem[E11]{E}
Weinan E.
\newblock {\em Principles of multiscale modeling}.
\newblock Cambridge University Press, Cambridge, 2011.

\bibitem[EH01]{Eckmann-Hairer}
J.-P. Eckmann and M.~Hairer.
\newblock Uniqueness of the invariant measure for a stochastic {PDE} driven by
  degenerate noise.
\newblock {\em Comm. Math. Phys.}, 219(3):523--565, 2001.

\bibitem[ELJL99]{Elworthy-LeJan-Li-book}
K.~D. Elworthy, Y.~Le~Jan, and Xue-Mei Li.
\newblock {\em On the geometry of diffusion operators and stochastic flows},
  volume 1720 of {\em Lecture Notes in Mathematics}.
\newblock Springer-Verlag, 1999.

\bibitem[ELJL10]{Elworthy-LeJan-Li-book-2}
K.~David Elworthy, Yves Le~Jan, and Xue-Mei Li.
\newblock {\em The geometry of filtering}.
\newblock Frontiers in Mathematics. Birkh\"auser Verlag, Basel, 2010.

\bibitem[Elw82]{Elworthy-book}
K.~D. Elworthy.
\newblock {\em Stochastic differential equations on manifolds}, volume~70 of
  {\em London Mathematical Society Lecture Note Series}.
\newblock Cambridge University Press, Cambridge, 1982.

\bibitem[{\'E}me89]{Emery}
Michel {\'E}mery.
\newblock {\em Stochastic calculus in manifolds}.
\newblock Universitext. Springer-Verlag, Berlin, 1989.
\newblock With an appendix by P.-A. Meyer.

\bibitem[FD11]{Fu-Duan}
Hongbo Fu and Jinqiao Duan.
\newblock An averaging principle for two-scale stochastic partial differential
  equations.
\newblock {\em Stoch. Dyn.}, 11(2-3):353--367, 2011.

\bibitem[FL11]{Fu-Liu}
Hongbo Fu and Jicheng Liu.
\newblock Strong convergence in stochastic averaging principle for two
  time-scales stochastic partial differential equations.
\newblock {\em J. Math. Anal. Appl.}, 384(1):70--86, 2011.

\bibitem[Fre78]{Freidlin78}
M.~I. Freidlin.
\newblock The averaging principle and theorems on large deviations.
\newblock {\em Uspekhi Mat. Nauk}, 33(5(203)):107--160, 238, 1978.

\bibitem[FW84]{Freidlin-Wentzell-original}
M.~I. Freidlin and A.~D. Wentzell.
\newblock {\em Random perturbations of dynamical systems}, volume 260 of {\em
  Grundlehren der Mathematischen Wissenschaften [Fundamental Principles of
  Mathematical Sciences]}.
\newblock Springer-Verlag, New York, 1984.
\newblock Translated from the Russian by Joseph Sz{\"u}cs.

\bibitem[FW12]{Freidlin-Wentzell}
Mark~I. Freidlin and Alexander~D. Wentzell.
\newblock {\em Random perturbations of dynamical systems}, volume 260 of {\em
  Grundlehren der Mathematischen Wissenschaften [Fundamental Principles of
  Mathematical Sciences]}.
\newblock Springer, Heidelberg, third edition, 2012.
\newblock Translated from the 1979 Russian original by Joseph Sz{\"u}cs.

\bibitem[GGR16]{Gonzales-Gargate-Ruffino}
Ivan~I. Gonzales-Gargate and Paulo~R. Ruffino.
\newblock An averaging principle for diffusions in foliated spaces.
\newblock {\em Ann. Probab.}, 44(1):567--588, 2016.

\bibitem[GW79]{Greene-Wu}
R.~E. Greene and H.~Wu.
\newblock {$C^{\infty }$}\ approximations of convex, subharmonic, and
  plurisubharmonic functions.
\newblock {\em Ann. Sci. \'Ecole Norm. Sup. (4)}, 12(1):47--84, 1979.

\bibitem[Has68]{hasminskii68}
R.~Z. Has{\cprime}minskii.
\newblock On the principle of averaging the {I}t\^o's stochastic differential
  equations.
\newblock {\em Kybernetika (Prague)}, 4:260--279, 1968.

\bibitem[Hel82]{Helland}
Inge~S. Helland.
\newblock Central limit theorems for martingales with discrete or continuous
  time.
\newblock {\em Scand. J. Statist.}, 9(2):79--94, 1982.

\bibitem[HM06]{Hairer-Mattingly}
Martin Hairer and Jonathan~C. Mattingly.
\newblock Ergodicity of the 2{D} {N}avier-{S}tokes equations with degenerate
  stochastic forcing.
\newblock {\em Ann. of Math. (2)}, 164(3):993--1032, 2006.

\bibitem[HMS11]{Hairer-Mattingly-Scheutzow}
M.~Hairer, J.~C. Mattingly, and M.~Scheutzow.
\newblock Asymptotic coupling and a general form of {H}arris' theorem with
  applications to stochastic delay equations.
\newblock {\em Probab. Theory Related Fields}, 149(1-2):223--259, 2011.

\bibitem[H{\"o}r67]{Hormander-acta}
Lars H{\"o}rmander.
\newblock Hypoelliptic second order differential equations.
\newblock {\em Acta Math.}, 119:147--171, 1967.

\bibitem[HP04]{Hairer-Pavliotis}
M.~Hairer and G.~A. Pavliotis.
\newblock Periodic homogenization for hypoelliptic diffusions.
\newblock {\em J. Statist. Phys.}, 117(1-2):261--279, 2004.

\bibitem[HP08]{Hairer-Pardoux}
Martin Hairer and Etienne Pardoux.
\newblock Homogenization of periodic linear degenerate {PDE}s.
\newblock {\em J. Funct. Anal.}, 255(9):2462--2487, 2008.

\bibitem[HP11]{Hairer-Pillai}
M.~Hairer and N.~S. Pillai.
\newblock Ergodicity of hypoelliptic {SDE}s driven by fractional {B}rownian
  motion.
\newblock {\em Ann. Inst. Henri Poincar\'e Probab. Stat.}, 47(2):601--628,
  2011.

\bibitem[HR15]{Hogele-Ruffino}
Michael H\"ogele and Paulo Ruffino.
\newblock Averaging along foliated {L}\'evy diffusions.
\newblock {\em Nonlinear Anal.}, 112:1--14, 2015.

\bibitem[IO90]{Ikeda-Ogura}
Nobuyuki Ikeda and Yukio Ogura.
\newblock A degenerating sequence of {R}iemannian metrics on a manifold and
  their {B}rownian motions.
\newblock In {\em Diffusion processes and related problems in analysis, {V}ol.\
  {I}}, volume~22 of {\em Progr. Probab.}, pages 293--312. Birkh\"auser Boston,
  1990.

\bibitem[Kif88]{KiferBook88}
Yuri Kifer.
\newblock {\em Random perturbations of dynamical systems}, volume~16 of {\em
  Progress in Probability and Statistics}.
\newblock Birkh\"auser Boston, Inc., Boston, MA, 1988.

\bibitem[KKM16]{Korepanov-Koslof-Melbourne}
A.~Korepanov, Z.~Kosloff, and I.~Melbourne.
\newblock Martingale-coboundary decomposition for families of dynamical
  systems.
\newblock preprint, 2016.

\bibitem[KLO12]{Komorowski-Landim-Olla}
Tomasz Komorowski, Claudio Landim, and Stefano Olla.
\newblock {\em Fluctuations in {M}arkov processes}, volume 345 of {\em
  Grundlehren der Mathematischen Wissenschaften [Fundamental Principles of
  Mathematical Sciences]}.
\newblock Springer, Heidelberg, 2012.
\newblock Time symmetry and martingale approximation.

\bibitem[KM17]{Kelly-Melbourne}
David Kelly and Ian Melbourne.
\newblock Deterministic homogenization for fast-slow systems with chaotic
  noise.
\newblock {\em J. Funct. Anal.}, 272(10):4063--4102, 2017.

\bibitem[KN63]{Kobayashi-NomizuI}
Shoshichi Kobayashi and Katsumi Nomizu.
\newblock {\em Foundations of differential geometry. {V}ol {I}}.
\newblock Interscience Publishers, a division of John Wiley \& Sons, New
  York-Lond on, 1963.

\bibitem[Kra40]{Kramers}
H.~A. Kramers.
\newblock Brownian motion in a field of force and the diffusion model of
  chemical reactions.
\newblock {\em Physica}, 7:284--304, 1940.

\bibitem[KS85]{Kusuoka-Stroock-II}
S.~Kusuoka and D.~Stroock.
\newblock Applications of the {M}alliavin calculus. {II}.
\newblock {\em J. Fac. Sci. Univ. Tokyo Sect. IA Math.}, 32(1):1--76, 1985.

\bibitem[Kur70]{Kurtz70}
Thomas~G. Kurtz.
\newblock A general theorem on the convergence of operator semigroups.
\newblock {\em Trans. Amer. Math. Soc.}, 148:23--32, 1970.

\bibitem[KV86]{Kipnis-Varadhan}
C.~Kipnis and S.~R.~S. Varadhan.
\newblock Central limit theorem for additive functionals of reversible {M}arkov
  processes and applications to simple exclusions.
\newblock {\em Comm. Math. Phys.}, 104(1):1--19, 1986.

\bibitem[Li08]{Li-averaging}
Xue-Mei Li.
\newblock An averaging principle for a completely integrable stochastic
  {H}amiltonian system.
\newblock {\em Nonlinearity}, 21(4):803--822, 2008.

\bibitem[Li12]{Li-OM-1}
Xue-Mei Li.
\newblock Effective diffusions with intertwined structures.
\newblock arxiv:1204.3250, 2012.

\bibitem[Li15]{Li-geodesic}
Xue-Mei Li.
\newblock Random perturbation to the geodesic equation.
\newblock {\em Annals of Probability}, 44(1):544--566, 2015.

\bibitem[Li16a]{Li-homogeneous}
Xue-Mei Li.
\newblock Homogenization on homogeneous spaces.
\newblock To appear in the Journal of the Mathematical Society of Japan, 2016.

\bibitem[Li16b]{Li-limits}
Xue-Mei Li.
\newblock Limits of random differential equations on manifolds.
\newblock {\em Probab. Theory Relat. Fields}, 166(3-4):659--712, 2016.
\newblock DOI 10.1007/s00440-015-0669-x.

\bibitem[LO12]{Liverani-Olla}
Carlangelo Liverani and Stefano Olla.
\newblock Toward the {F}ourier law for a weakly interacting anharmonic crystal.
\newblock {\em J. Amer. Math. Soc.}, 25(2):555--583, 2012.

\bibitem[MM90]{Mazzeo-Melrose}
Rafe~R. Mazzeo and Richard~B. Melrose.
\newblock The adiabatic limit, {H}odge cohomology and {L}eray's spectral
  sequence for a fibration.
\newblock {\em J. Differential Geom.}, 31(1):185--213, 1990.

\bibitem[MS39]{Myers-Steenrod}
S.~B. Myers and N.~E. Steenrod.
\newblock The group of isometries of a {R}iemannian manifold.
\newblock {\em Ann. of Math. (2)}, 40(2):400--416, 1939.

\bibitem[Nel67]{Nelson}
Edward Nelson.
\newblock {\em Dynamical theories of {B}rownian motion}.
\newblock Princeton University Press, Princeton, N.J., 1967.

\bibitem[OT96]{Ogura-Taniguchi96}
Yukio Ogura and Setsuo Taniguchi.
\newblock A probabilistic scheme for collapse of metrics.
\newblock {\em J. Math. Kyoto Univ.}, 36(1):73--92, 1996.

\bibitem[PK74]{Kohler-Papanicolaou74}
G.~C. Papanicolaou and W.~Kohler.
\newblock Asymptotic theory of mixing stochastic ordinary differential
  equations.
\newblock {\em Comm. Pure Appl. Math.}, 27:641--668, 1974.

\bibitem[PS08]{Pavliotis-Stuart08}
Grigorios~A. Pavliotis and Andrew~M. Stuart.
\newblock {\em Multiscale methods}, volume~53 of {\em Texts in Applied
  Mathematics}.
\newblock Springer, New York, 2008.
\newblock Averaging and homogenization.

\bibitem[PSV77]{Papanicolaou-Stroock-Varadhan77}
G.~C. Papanicolaou, D.~Stroock, and S.~R.~S. Varadhan.
\newblock Martingale approach to some limit theorems.
\newblock In {\em Papers from the {D}uke {T}urbulence {C}onference ({D}uke
  {U}niv.,1976)}, pages ii+120 pp. Duke Univ., Durham, N.C., 1977.

\bibitem[PV73]{Papanicolaou-Stroock-Varadhan73}
G.~C. Papanicolaou and S.~R.~S. Varadhan.
\newblock A limit theorem with strong mixing in {B}anach space and two
  applications to stochastic differential equations.
\newblock {\em Comm. Pure Appl. Math.}, 26:497--524, 1973.

\bibitem[Ruf15]{Ruffino}
Paulo~R. Ruffino.
\newblock Application of an averaging principle on foliated diffusions:
  topology of the leaves.
\newblock {\em Electron. Commun. Probab.}, 20:no. 28, 5, 2015.

\bibitem[SHS02]{Skorohod-Hoppensteadt-Habib}
Anatoli~V. Skorokhod, Frank~C. Hoppensteadt, and Habib Salehi.
\newblock {\em Random perturbation methods with applications in science and
  engineering}, volume 150 of {\em Applied Mathematical Sciences}.
\newblock Springer-Verlag, New York, 2002.

\bibitem[Str61]{Stratonovich61}
R.~L. Stratonovich.
\newblock Selected problems in the theory of fluctuations in radio engineering.
\newblock {\em Sov. Radio, Moscow}, 1961.
\newblock In Russian.

\bibitem[Str63]{Stratonovich63}
R.~L. Stratonovich.
\newblock {\em Topics in the theory of random noise. {V}ol. {I}: {G}eneral
  theory of random processes. {N}onlinear transformations of signals and
  noise}.
\newblock Revised English edition. Translated from the Russian by Richard A.
  Silverman. Gordon and Breach Science Publishers, New York-London, 1963.

\bibitem[SV10]{Stroock-Varadhan}
D.~Stroock and S.~R.~S. Varadhan.
\newblock Theory of diffusion processes.
\newblock In {\em Stochastic differential equations}, volume~77 of {\em
  C.I.M.E. Summer Sch.}, pages 149--191. Springer, Heidelberg, 2010.

\bibitem[SY94]{Schoen-Yau}
R.~Schoen and S.-T. Yau.
\newblock {\em Lectures on differential geometry}.
\newblock Conference Proceedings and Lecture Notes in Geometry and Topology, I.
  International Press, Cambridge, MA, 1994.
\newblock Lecture notes prepared by Wei Yue Ding, Kung Ching Chang [Gong Qing
  Zhang], Jia Qing Zhong and Yi Chao Xu, Translated from the Chinese by Ding
  and S. Y. Cheng, Preface translated from the Chinese by Kaising Tso.

\bibitem[Tam10]{Tam-exhaustion}
Luen-Fai Tam.
\newblock Exhaustion functions on complete manifolds.
\newblock In {\em Recent advances in geometric analysis}, volume~11 of {\em
  Adv. Lect. Math. (ALM)}, pages 211--215. Int. Press, Somerville, MA, 2010.

\bibitem[Tan79]{Tanno-79}
Sh{\^u}kichi Tanno.
\newblock The first eigenvalue of the {L}aplacian on spheres.
\newblock {\em T\^ohoku Math. J. (2)}, 31(2):179--185, 1979.

\bibitem[UO30]{Ornstein-Uhlenbeck}
G.~E. Uhlenbeck and L.~S. Ornstein.
\newblock Brownian motion in a field of force and the diffusion model of
  chemical reactions.
\newblock {\em Physical Review}, 36:823--841, 1930.

\bibitem[Ura86]{Urakawa86}
Hajime Urakawa.
\newblock The first eigenvalue of the {L}aplacian for a positively curved
  homogeneous {R}iemannian manifold.
\newblock {\em Compositio Math.}, 59(1):57--71, 1986.

\bibitem[vE10]{Erp-index}
Erik van Erp.
\newblock The {A}tiyah-{S}inger index formula for subelliptic operators on
  contact manifolds. {P}art {I}.
\newblock {\em Ann. of Math. (2)}, 171(3):1647--1681, 2010.

\bibitem[vE11]{Erp-index-foliated}
Erik van Erp.
\newblock The index of hypoelliptic operators on foliated manifolds.
\newblock {\em J. Noncommut. Geom.}, 5(1):107--124, 2011.

\bibitem[Ver90]{Veretennikov}
A.~Yu. Veretennikov.
\newblock On an averaging principle for systems of stochastic differential
  equations.
\newblock {\em Mat. Sb.}, 181(2):256--268, 1990.

\bibitem[Yos65]{Yosida}
K\^osaku Yosida.
\newblock {\em Functional analysis}.
\newblock Die Grundlehren der Mathematischen Wissenschaften, Band 123. Academic
  Press, Inc., New York; Springer-Verlag, Berlin, 1965.

\end{thebibliography}

\end{document}